\documentclass{amsart}
\usepackage[margin=1.5in]{geometry}
\usepackage{amsfonts}
\usepackage{lineno,hyperref}
\usepackage{amsmath}
\usepackage{amssymb}
\usepackage{amsthm}
\usepackage{cancel}
\usepackage{tikz-cd}
\usepackage{arydshln}
\usepackage{graphicx}
\usepackage{algorithm}
\usepackage[noend]{algpseudocode}

\newcommand{\RR}{\mathbb{R}}

\newcommand{\NN}{\mathbb{N}}

\newcommand{\ZZ}{\mathbb{Z}}
\newcommand{\QQ}{\mathbb{Q}}
\newcommand{\OO}{\mathcal{O}}
\newcommand{\HH}{\mathbb{H}}
\newcommand{\nrm}{\text{nrm}}
\newcommand{\tr}{\text{tr}}
\newcommand{\disc}{\text{disc}}

\newcommand{\Mat}{\text{Mat}}

\newtheorem{definition}{Definition}[section]
\newtheorem{theorem}{Theorem}[section]
\newtheorem{corollary}{Corollary}[section]
\newtheorem{lemma}{Lemma}[section]

\theoremstyle{remark}
\newtheorem{problem}{Open Problem}
\newtheorem{alg}{Algorithm}[section]
\newtheorem{remark}{Remark}[section]

\begin{document}
\title{The Twisted Euclidean Algorithm: Applications to Number Theory and Geometry}
\author{Arseniy (Senia) Sheydvasser}
\address{Department of Mathematics, Graduate Center at CUNY, 365 5th Ave, New York, NY 10016}
\email{ssheydvasser@gc.cuny.edu}


\date{\today}


\begin{abstract}
We introduce a generalization of the Euclidean algorithm for rings equipped with an involution, and completely enumerate all isomorphism classes of orders over definite, rational quaternion algebras equipped with an orthogonal involution that admit such an algorithm. We give two applications: first, any order that admits such an algorithm has class number $1$; second, we show how the existence of such an algorithm relates to the problem of constructing explicit Dirichlet domains for Kleinian subgroups of the isometry group of hyperbolic $4$-space.
\end{abstract}

\maketitle

\section{Introduction:}

For an algorithm that is thousands of years old, the Euclidean algorithm spawns new insights and new questions with remarkable regularity. Our present goal is to give an analog of the Euclidean algorithm that applies to any ring equipped with an involution---we shall demonstrate that many of the properties that are true of Euclidean rings continue to hold in this new setting, even though this new analog seems to apply to a strictly larger class of rings. We begin with some definitions.
    
\begin{definition}[See \cite{Involutions}]
An \emph{involution} on a ring $R$ is a map $\sigma: R \rightarrow R$ such that for all $x,y \in R$,
    \begin{enumerate}
        \item $\sigma(x + y) = \sigma(x) + \sigma(y)$,
        \item $\sigma(xy) = \sigma(y) \sigma(x)$, and
        \item $\sigma(\sigma(x)) = x$.
    \end{enumerate}
    
\noindent Any such ring has two important subsets.
    \begin{align*}
        R^+ := \left\{r\in R\middle|\sigma(r) = r\right\} \\
        R^- := \left\{r \in R \middle|\sigma(r) = -r\right\}.
    \end{align*}
    
\noindent A \emph{homomorphism of} ring with involution $\varphi: (R,\sigma_1) \rightarrow (S,\sigma_2)$ will be a ring homomorphism $\varphi: R \rightarrow S$ such that $\sigma_2 \circ \varphi = \varphi \circ \sigma_1$.
\end{definition}

Rings with involution are extremely common. We list a few basic examples:
    \begin{enumerate}
        \item If $R$ is a commutative ring, then one can simply take $\sigma$ to be the identity function.
        \item If $K$ if a Galois field extension of $F$ and $\sigma \in \text{Gal}(K/F)$ has order two, then $(K,\sigma)$ is a ring with involution.
        \item The Hamilton quaternions $H_\RR$ have two commonly referenced involutions:
            \begin{align*}
                \overline{x + yi + zj + tk} &= x - yi - zj - tk \\
                (x + yi + zj + tk)^\ddagger &= x + yi + zj - tk.
            \end{align*}
    \end{enumerate}

In this setting, we can define a generalization of Euclidean rings which makes use of the additional structure afforded by the involution.

\begin{definition}
Let $(R,\sigma)$ be a ring with involution and without zero divisors. Suppose there exists a well-ordered set $W$ and a function $\Phi: R \rightarrow W$ such that for all $a,b \in R$ such that $b \neq 0$ and $a\sigma(b) \in R^+$, there exists $q \in R^+$ such that $\Phi(a - bq) < \Phi(b)$. Then we say that $R$ is a $\sigma$-\emph{Euclidean ring}, with \emph{stathm}\footnote{The etymology of the word ``stathm" is a bit of a mystery to me. The earliest paper that I could track down that makes use of the word is due to Wedderburn, from 1931 \cite{Wedderburn1931}, but I suspect that it is far older.} $\Phi$.
\end{definition}

First, note that this definition includes the usual definition of Euclidean domains as a special case---simply take $\sigma$ to be the identity map. However, in the non-commutative case, these definitions diverge and we shall prove the existence of $\sigma$-Euclidean rings that are not Euclidean. On the other hand, we claim that $\sigma$-Euclidean rings are not as ad hoc as they might at first appear. In the first place, it is easy to see that the property of being $\sigma$-Euclidean is preserved by isomorphisms of rings with involution. Secondly, the usual Euclidean algorithm can be viewed as computing the GCD of a pair $(a,b) \in R^2$ and constructing a matrix with that as the top row---if the GCD of $a,b$ happens to be $1$, then this matrix is in $SL(2,R)$. Just so, we shall show that there is a twisted Euclidean algorithm that comptues the GCD of a pair $(a,b)$ and constructs a matrix with that as the top row---if the GCD of $a,b$ happens to be $1$, then this matrix is inside a group $SL^\sigma(2,R)$ which we shall define below. Additionally, we give two motivating examples of applications of $\sigma$-Euclidean rings.

\subsection{The Class Number 1 Problem:}

One of the most important properties of Euclidean domains is that they are always principal ideal domains; indeed, establishing an analog of the Euclidean algorithm is the oldest method of proving that a ring is a unique factorization domain. A similar result is true for $\sigma$-Euclidean rings: every invertible ideal that is generated by two elements $a,b$ such that $a\sigma(b) \in R^+$ is principal if $R$ is $\sigma$-Euclidean---see Corollary \ref{Almost Principal Ring}. There may be a large class of rings where every invertible ideal is so generated, but we provide one specific example: if $R$ is an order of a quaternion algebra $H$ over an algebraic number field $K$ and $\sigma$ extends to an orthogonal involution on $H$. This has the following important consequence.

\begin{theorem}\label{Class Number 1 Theorem}
Let $H$ be a quaternion algebra over an algebraic number field $K$ with orthogonal involution $\sigma$. If $\OO$ is an order of $H$ and a Euclidean ring, then $\OO$ has right class number $1$.
\end{theorem}

We will make use of this result to enumerate all $\sigma$-Euclidean rings that arise as orders of rational, definite quaternion algebras equipped with an orthogonal involution.

\begin{theorem}\label{Enumeration Theorem}
Let $H$ be a rational, definite quaternion algebra with orthogonal involution $\sigma$. Let $\OO$ be an order of $H$. It is a $\sigma$-Euclidean ring if and only if it is isomorphic to one of the rings with involution listed in Table \ref{All Euclidean Orders Short}.
\end{theorem}

    \begin{table}
        \begin{minipage}{.45\textwidth}
    \begin{align*}
    \begin{array}{ll}
    \ZZ \oplus \ZZ i \oplus \ZZ j \oplus \ZZ \frac{1 + i + j + ij}{2} &\subset \left(\frac{-1,-1}{\QQ}\right) \\
    \ZZ \oplus \ZZ i \oplus \ZZ j \oplus ij &\subset \left(\frac{-1,-1}{\QQ}\right) \\
    \ZZ \oplus \ZZ i \oplus \ZZ \frac{1 + i + j}{2} \oplus \ZZ \frac{1 + i + ij}{2} &\subset \left(\frac{-1,-2}{\QQ}\right) \\
    \ZZ \oplus \ZZ i \oplus \ZZ \frac{i + j}{2} \oplus \ZZ \frac{2 + ij}{4} &\subset \left(\frac{-2,-6}{\QQ}\right) \\
    \ZZ \oplus \ZZ i \oplus \ZZ \frac{i + j}{2} \oplus \ZZ \frac{ij}{2} &\subset \left(\frac{-2,-6}{\QQ}\right) \\
    \ZZ \oplus \ZZ i \oplus \ZZ \frac{1 + j}{2} \oplus \ZZ \frac{i + ij}{2} &\subset \left(\frac{-2,-3}{\QQ}\right) \\
    \ZZ \oplus \ZZ i \oplus \ZZ \frac{1 + i + j}{2} \oplus \ZZ \frac{1 + i + ij}{2} &\subset \left(\frac{-1,-10}{\QQ}\right) \\
    \ZZ \oplus \ZZ i \oplus \ZZ \frac{i + j}{2} \oplus \ZZ \frac{1 + ij}{2} &\subset \left(\frac{-1,-3}{\QQ}\right)
    \end{array}
    \end{align*}
    \end{minipage}%
    \begin{minipage}{.45\textwidth}
    \begin{align*}
    \begin{array}{ll}
    \ZZ \oplus \ZZ i \oplus \ZZ \frac{1 + j}{2} \oplus \ZZ \frac{i + k}{2} &\subset \left(\frac{-1,-3}{\QQ}\right) \\
    \ZZ \oplus \ZZ i \oplus \ZZ \frac{1 + i + j}{2} \oplus \ZZ \frac{1 + i + ij}{2} &\subset \left(\frac{-1,-6}{\QQ}\right) \\
    \ZZ \oplus \ZZ i \oplus \ZZ \frac{2 + i + j}{4} \oplus \ZZ \frac{2 + 2i + ij}{4} &\subset \left(\frac{-2,-10}{\QQ}\right) \\
    \ZZ \oplus \ZZ i \oplus \ZZ \frac{1 + i + j}{2} \oplus \ZZ \frac{i + ij}{2} &\subset \left(\frac{-2,-5}{\QQ}\right) \\
    \ZZ \oplus \ZZ i \oplus \ZZ \frac{i + j}{2} \oplus \ZZ \frac{1 + ij}{2} &\subset \left(\frac{-1,-7}{\QQ}\right) \\
    \ZZ \oplus \ZZ i \oplus \ZZ \frac{1 + j}{2} \oplus \ZZ \frac{i + k}{2} &\subset \left(\frac{-1,-7}{\QQ}\right) \\
    \ZZ \oplus \ZZ i \oplus \ZZ \frac{2 + i + j}{4} \oplus \ZZ \frac{2 + 2i + ij}{4} &\subset \left(\frac{-2,-26}{\QQ}\right)
    \end{array}
    \end{align*}
    \end{minipage}

    \caption{All $\sigma$-Euclidean rings that arise as orders of rational, definite quaternion algebras equipped with an orthogonal involution, up to isomorphism (as rings with involution). The involution is always taken to be $\sigma(z) = (ij)\overline{z}(ij)^{-1}$, where $\overline{z}$ denotes the quaternion conjugate of $z$.}
    \label{All Euclidean Orders Short}
    \end{table}

\subsection{Efficient Construction of Dirichlet Domains:}

If $R$ is a Euclidean domain, then it follows that $SL(2,R)$ is generated by elementary matrices. It turns out that something very similar is true for rings with involution. There is a corresponding group $SL^\sigma(2,R)$, which can be defined as follows. Let $(R,\sigma)$ be any ring with involution. Then $\Mat(2,R)$, the collection of $2\times 2$ matrices with coefficients in $R$, has an involution defined on it by
    \begin{align*}
        \hat{\sigma}: \Mat(2,R) &\rightarrow \Mat(2,R) \\
        \begin{pmatrix} a & b \\ c & d \end{pmatrix} &\mapsto \begin{pmatrix} \sigma(d) & -\sigma(b) \\ -\sigma(c) & \sigma(a) \end{pmatrix}.
    \end{align*}
    
\noindent Then there is a corresponding group---which we shall call the twisted special linear group---
    \begin{align*}
    SL^\sigma(2,R) := \left\{M \in \Mat(2,R)\middle|M\hat{\sigma}(M) = \hat{\sigma}(M)M = 1\right\}
    \end{align*}
    
\noindent which was defined in \cite{Sheydvasser2020algebraic}. An important subgroup is $E^\sigma(2,R)$, generated by elements of the form
    \begin{align*}
    \begin{pmatrix} 1 & \tau \\ 0 & 1 \end{pmatrix}, \begin{pmatrix} u & 0 \\ 0 & \sigma(u)^{-1} \end{pmatrix}, \begin{pmatrix} 0 & 1 \\ -1 & 0 \end{pmatrix},
    \end{align*}
    
\noindent where $\tau \in R^+$ and $u \in R^\times$. If $R$ is $\sigma$-Euclidean, then $SL^\sigma(2,R) = E^\sigma(2,R)$---see Corollary \ref{Generation of SL}. Such groups appear in various contexts. If $R$ is commutative and $\sigma$ is the identity map, then $SL^\sigma(2,R) = SL(2,R)$. A different example due to work of Vahlen \cite{Vahlen1902} is that if you take $H_\RR$ to be the standard Hamilton quaternions and $\ddagger$ an orthogonal involution on $H_\RR$, then $SL^\ddagger(2,H_\RR)/\left\{\pm id\right\} \cong \text{Isom}^0(\HH^4) \cong SO^+(4,1)$, the orientation-preserving isometry group of hyperbolic $4$-space. It was also demonstrated in \cite{Sheydvasser2020algebraic} that if $F$ is an algebraically closed field with characteristic not $0$, $A$ is a central simple algebra over $F$, and $\sigma$ is an $F$-linear involution on $A$, then $SL^\sigma(2,A)$ is either a symplectic or orthogonal group, depending on the dimension of $A^+$. If $\OO$ is an order closed under $\sigma$, then $SL^\sigma(2,\OO)$ is an arithmetic subgroup of $SL^\sigma(2,A)$. In particular, if we restrict to the case where $F = \QQ$, $H$ is a definite, rational quaternion algebra, and $\sigma$ is not quaternion conjugation, then $SL^\sigma(2,\OO)/\{\pm id\}$ is a Kleinian subgroup of $\text{Isom}^0(\HH^4)$. Thus, we can define a quotient orbifold $\HH^4/SL^\sigma(2,\OO)$ and study its topological and geometric properties in terms of the algebraic properties of $\OO$. If $\OO$ happens to be $\sigma$-Euclidean, then this is particularly easy to do.

\begin{theorem}\label{Main Theorem}
Let $H$ be a definite, rational quaternion algebra with orthogonal involution $\ddagger$. Let $\OO$ be an order of $H$ that is also a Euclidean $\ddagger$-ring. There exists an algorithm to compute a Dirichlet domain for the group $SL^\ddagger(2,\OO)$ acting on $\HH^4$. The corresponding quotient orbifold $\HH^4/SL^\ddagger(2,\OO)$ has exactly one cusp.
\end{theorem}

\section{The Twisted Euclidean Algorithm:}

In subsequent sections, we shall restrict to the special case where our rings will be orders of quaternion algebras and the involution is orthogonal. To start with, however, we shall define the twisted Euclidean algorithm and show consequences that apply to all rings with involution. So, given a stathm $\Phi$ for a $\sigma$-Euclidean ring $R$, let $f_\Phi: R \times R \rightarrow R \times R$ be a function such that if $a,b \in R$, $b \neq 0$, and $a\sigma(b) \in R^+$, then $f_\Phi(x,y) = (q,a - bq)$ such that $q \in R^+$ and $\Phi(a - bq) < \Phi(b)$.

\begin{alg}\label{Euclidean Algorithm}
On an input of $a,b \in R$ such that $a\sigma(b) \in R^+$, this algorithm returns finite sequences $r_i, s_i, t_i \subset R$ with $0 \leq i \leq k + 1$ satisfying the following properties.
    \begin{enumerate}
        \item $r_{k + 1} = 0$.
        \item $r_k$ is a right GCD of $a$ and $b$.
        \item $a\sigma(s_i) - b\sigma(t_i) = r_i$ for all $i$, and $s_i \sigma(t_i) \in R^+$.
        \item For all $i$,
            \begin{align*}
            \begin{pmatrix} t_i & s_i \\ (-1)^i t_{i + 1} & (-1)^i s_{i + 1} \end{pmatrix} \in E^\sigma(2,R).
            \end{align*}
    \end{enumerate}
\begin{algorithm}[H]
\begin{algorithmic}[1]
\Procedure{TwistedEuclideanAlg}{$x$,$y$}
\State $r_\text{list} \gets [a,b]$ \Comment We initialize lists $r_\text{list}, s_\text{list}, t_\text{list}$.
\State $s_\text{list} \gets [1,0]$
\State $t_\text{list} \gets [0,-1]$
\While{$r_\text{list}[-1] \neq 0$} \Comment Here $l[-1]$ denotes the last element of $l$.
    \State $r_{i - 2} \gets r_\text{list}[-2]$ \Comment Here $l[-2]$ denotes the second to last element of $l$.
    \State $r_{i - 1} \gets r_\text{list}[-1]$
    \State $(q, r_i) \gets f_\Phi(r_{i - 2}, r_{i - 1})$
    \State $\text{append}(r_\text{list},r_i)$ \Comment Here $\text{append}(l,x)$ denotes appending $x$ to the end of $l$.
    \State $s_{i - 2} \gets s_\text{list}[-2]$
    \State $s_{i - 1} \gets s_\text{list}[-1]$
    \State $s_i \gets s_{i - 2} - q s_{i - 1}$
    \State $\text{append}(s_\text{list},s_i)$
    \State $t_{i - 2} \gets t_\text{list}[-2]$
    \State $t_{i - 1} \gets t_\text{list}[-1]$
    \State $t_i \gets t_{i - 2} - q t_{i - 1}$
    \State $\text{append}(s_\text{list},t_i)$
\EndWhile
\State \textbf{return} $r_\text{list}, s_\text{list}, t_\text{list}$
\EndProcedure
\end{algorithmic}
\end{algorithm}
\end{alg}

\begin{remark}
With some minor differences, this algorithm appeared in the author's previous work on sphere packings \cite{Sheydvasser2019} in the context of orders of definite, rational quaternion algebras and $\Phi = \nrm$; in that context, it was connected to the question of whether the sphere packing is connected or not---this idea is itself adapted from Katherine Stange's work on circle packings \cite{Stange2017,StangeFuture}. The proof of correctness is essentially the same, but is reproduced both for convenience and to accommodate the more slightly general setting.
\end{remark}

\begin{remark}
If $R$ is a commutative ring and $\sigma$ is the identity function, then this is nothing more than the usual Euclidean algorithm.
\end{remark}

\begin{remark}
Technically, Algorithm \ref{Euclidean Algorithm} is an algorithm if and only if $f_\Phi$ is computable; otherwise, it is only a semi-algorithm.
\end{remark}

\begin{proof}[Proof of Correctness]
Note that, by construction, $\Phi(r_i) < \Phi(r_{i - 1})$ for all $i \geq 2$---since $\Phi$ maps to a well-ordered set, the sequence of $s_i$'s must be finite. This can only happen if it is eventually zero. Therefore, eventually $r_{k + 1} = 0$ for some large enough $k$, and the algorithm halts. To show that $r_k$ is then a right GCD of $a$ and $b$, we proceed by induction---specifically, we first prove that $r_k$ divides $r_i$ for all $i < k$. The base case follows from the observation that $0 = r_{k + 1} = r_{k - 1} - r_kq$, whence $r_{k - 1} = r_k q$. For all other $i$, we note that $r_i = r_{i + 1} - r_{i + 2}q$ for some $q \in R^+$, and since $r_k$ divides $r_{i + 1}$ and $r_{i + 2}$ by assumption, it divides $r_i$. Thus, $r_k$ divides $a$ and $b$. On the other hand, if $r \in R$ divides $a$ and $b$, then $r$ divides $r_2 = r_1 - r_0 q$, from which it follows by induction that it divides $r_i$ for all $i$. We conclude that $r_k$ is a right GCD of $a$ and $b$.

To prove that $a\sigma(s_i) - b\sigma(t_i) = r_i$ for all $i \leq k$, we again proceed by induction. It is obvious for $i = 0$ and $i = 1$. For all other $i$, we note that $s_i = s_{i - 2} - qs_{i - 1}$ and $t_i = t_{i - 2} - qt_{i - 1}$ for some $q \in R^+$, and therefore
    \begin{align*}
    a\sigma(s_i) - b\sigma(t_i) &= a\sigma\left(s_{i - 2} - qs_{i - 1}\right) - b\sigma\left(t_{i - 2} - qt_{i - 1}\right) \\
    &= a\sigma(s_{i - 2}) - b\sigma(t_{i - 2}) - \left(a\sigma(s_{i - 1}) - b\sigma(t_{i - 1})\right)q \\
    &= r_{i - 2} - r_{i - 1}q \\
    &= r_i.
    \end{align*}
    
\noindent On the other hand,
    \begin{align*}
    \begin{pmatrix} t_0 & s_0 \\ t_1 & s_1 \end{pmatrix} = \begin{pmatrix} 0 & 1 \\ -1 & 0 \end{pmatrix} \in E^\ddagger(2,R)
    \end{align*}
    
\noindent and
    \begin{align*}
    \begin{pmatrix} t_{i - 1} & s_{i - 1} \\ -t_{i} & -s_{i} \end{pmatrix} &= \underbrace{\begin{pmatrix} 0 & 1 \\ -1 & q \end{pmatrix}}_{\in E^\sigma(2,R)}\begin{pmatrix} t_{i - 2} & s_{i - 2} \\ t_{i - 1} & s_{i - 1} \end{pmatrix}, \\
    \begin{pmatrix} t_{i - 1} & s_{i - 1} \\ t_{i} & s_{i} \end{pmatrix} &= \underbrace{\begin{pmatrix} 0 & -1 \\ 1 & q \end{pmatrix}}_{\in E^\sigma(2,R)}\begin{pmatrix} t_{i - 2} & s_{i - 2} \\ -t_{i - 1} & -s_{i - 1} \end{pmatrix},
    \end{align*}
    
\noindent from which we get that
    \begin{align*}
    \begin{pmatrix} t_i & s_i \\ (-1)^i t_{i + 1} & (-1)^i s_{i + 1} \end{pmatrix} \in E^\sigma(2,R)
    \end{align*}
    
\noindent for all $i$, and from which it immediately follows that $s_i \sigma(t_i) \in R^+$.
\end{proof}

From the existence of this algorithm, we get two immediate consequences.

\begin{corollary}\label{Almost Principal Ring}
Let $R$ be a $\sigma$-Euclidean ring. Every invertible right ideal of $R$ that is generated by two elements $a,b$ such that $a\sigma(b) \in R^+$ is principal.
\end{corollary}

\begin{proof}
Choose any right invertible ideal $I$ of $R$ such that $I = a R + b R$ where $a\sigma(b) \in R^+$. If $b = 0$, we are done; otherwise, apply Algorithm \ref{Euclidean Algorithm} to $a,b$ to produce a GCD $g$ of $a$ and $b$---then $I = gR$.
\end{proof}

\begin{corollary}\label{Generation of SL}
Let $R$ be a $\sigma$-Euclidean ring. Then $SL^\sigma(2,R) = E^\sigma(2,R)$.
\end{corollary}

\begin{proof}
Choose any element
    \begin{align*}
    \begin{pmatrix} a & b \\ c & d \end{pmatrix} \in SL^\sigma(2,R).
    \end{align*}
    
\noindent If $b = 0$, we see that
    \begin{align*}
        \begin{pmatrix} a & 0 \\ c & d \end{pmatrix}\begin{pmatrix} \sigma(d) & 0 \\ -\sigma(c) & \sigma(a) \end{pmatrix} &= \begin{pmatrix} a\sigma(d) & * \\ * & * \end{pmatrix} = \begin{pmatrix} 1 & 0 \\ 0 & 1 \end{pmatrix} \\
        \begin{pmatrix} \sigma(d) & 0 \\ -\sigma(c) & \sigma(a) \end{pmatrix}\begin{pmatrix} a & 0 \\ c & d \end{pmatrix} &= \begin{pmatrix} \sigma(d)a & * \\ * & * \end{pmatrix} = \begin{pmatrix} 1 & 0 \\ 0 & 1 \end{pmatrix}
    \end{align*}
    
\noindent hence $a,d \in R^\times$. Therefore,
    \begin{align*}
    \begin{pmatrix} a & 0 \\ c & d \end{pmatrix} &= \begin{pmatrix} a & 0 \\ 0 & d \end{pmatrix} \begin{pmatrix} 1 & 0 \\ d^{-1}c & 1 \end{pmatrix} \\
    &= \begin{pmatrix} -a & 0 \\ 0 & -d \end{pmatrix} \begin{pmatrix} 0 & 1 \\ -1 & 0 \end{pmatrix} \begin{pmatrix} 1 & -d^{-1}c \\ 0 & 1 \end{pmatrix}\begin{pmatrix} 0 & 1 \\ -1 & 0 \end{pmatrix} \in E^\sigma(2,R).
    \end{align*}
    
\noindent Otherwise, apply Algorithm \ref{Euclidean Algorithm} to $a,b$, so we get that
    \begin{align*}
    \begin{pmatrix} t_k & s_k \\ (-1)^k t_{k + 1} & (-1)^k s_{k + 1} \end{pmatrix} \in E^\sigma(2,R)
    \end{align*}
    
\noindent with $a\sigma(s_i) - b\sigma(t_i) = r_i$ and $r_{k + 1} = 0$. Therefore,
    \begin{align*}
    \begin{pmatrix} a & b \\ c & d \end{pmatrix} \begin{pmatrix} t_k & s_k \\ (-1)^k t_{k + 1} & (-1)^k s_{k + 1} \end{pmatrix}^{-1} &= \begin{pmatrix} a & b \\ c & d \end{pmatrix} \begin{pmatrix} (-1)^k \sigma(s_{k + 1}) & -\sigma(s_k) \\ -(-1)^k \sigma(t_{k + 1}) & \sigma(t_k) \end{pmatrix} \\
    &= \begin{pmatrix} 0 & -r_k \\ * & * \end{pmatrix} \\
    &= \begin{pmatrix} 0 & -1 \\ 1 & 0 \end{pmatrix}\begin{pmatrix} * & * \\ 0 & r_k \end{pmatrix}.
    \end{align*}
    
\noindent Since this final matrix must be in $SL^\sigma(2,R)$, we see that in fact $r_k\sigma(r_k) = \sigma(r_k)r_k = 1$, and so for some $q \in R^+$, we have a decomposition
    \begin{align*}
    \begin{pmatrix} a & b \\ c & d \end{pmatrix} &= \begin{pmatrix} t_k & s_k \\ (-1)^k t_{k + 1} & (-1)^k s_{k + 1} \end{pmatrix}\begin{pmatrix} 0 & -1 \\ 1 & 0 \end{pmatrix}\begin{pmatrix} \sigma(r_k)^{-1} & 0 \\ 0 & r_k\end{pmatrix}\begin{pmatrix} 1 & q \\ 0 & 1 \end{pmatrix} \in E^\sigma(2,R)
    \end{align*}
    
\noindent as desired.
\end{proof}

\section{Basic Notions of Quaternion Algebras with Involution:}

We shall now review basic definitions and results about quaternion algebras equipped with an orthogonal involution. First, we recall the definition of an orthogonal involution.

\begin{definition}[See \cite{Involutions}]
Given a central simple algebra $A$ over a field $F$, an \emph{involution of the first kind} is an involution $\sigma: A \rightarrow A$ which is also an $F$-algebra homomorphism.
\end{definition}

All involutions of the first kind are either \emph{symplectic} or \emph{orthogonal} depending on whether an associated bilinear form is alternating or symmetric---rather than giving the general definition, we will concentrate on the special case of quaternion algebras, where the general theory is especially simple. Recall that a \emph{quaternion algebra} $H$ is a central simple algebra over a field $F$ of degree $2$. If $F$ does not have characteristic $2$, then one can find a standard basis $1,i,j,ij$ for $H$ such that $i^2 = a$, $j^2 = b$, and $ij = -ji$ for some $a,b \in F^\times$. It is standard to write
    \begin{align*}
    \left(\frac{a,b}{F}\right)
    \end{align*}
    
\noindent to denote the quaternion algebra over $F$ with this basis. Involutions of the first kind on quaternion algebras are very restricted.

\begin{theorem}[See \cite{Involutions}]\label{Classifying Involutions}
Let $H$ be a quaternion algebra over a field $F$. The only involutions of the first kind on $H$ are
    \begin{enumerate}
        \item the standard involution $z \mapsto \overline{z}$, also known as quaternion conjugation (this is the unique symplectic involution) and
        \item involutions $z \mapsto u\overline{z}u^{-1}$ where $u \in H^\times$ such that $u^2 \in F$ (these are the orthogonal involutions).
    \end{enumerate}
    
\noindent Given an orthogonal involution $z \mapsto z^\ddagger$, if $F$ does not have characteristic $2$, one can choose a standard basis $1,i,j,k$ for $H$ such that $i^2 = a$, $j^2 = b$ for some $a,b \in F^\times$, $ij = -ji$, and $(w + xi + yj + zk)^\ddagger = w + xi + yj -zk$.
\end{theorem}

\begin{remark}
We shall consistently use the notation $z \mapsto z^\ddagger$ to denote orthogonal involutions on quaternion algebras. This choice is motivated by the fact that the adjugate map $\Mat(2,F)$ is typically denoted by $\dagger$---however, this is just the standard involution on the quaternion algebra $\Mat(2,F)$.
\end{remark}

Ultimately, we shall be interested in studying orders that are closed under involutions. For a variety of reasons, the standard involution is far more studied in this context than the orthogonal involutions. The standard involution is useful in defining the (reduced) norm $\nrm(z) = z\overline{z}$ and the (reduced) trace $\tr(z) = z + \overline{z}$, for instance; one can classify quaternion algebras $H$ up to isomorphism by considering the quadratic form $z \mapsto \nrm(z)$ restricted to $H^0$, the subspace of $H$ with trace $0$. However, for our purposes, orthogonal involutions will be much more important. To see why, note that $H = H^+ \oplus H^-$ for any involution. Furthermore, from Theorem \ref{Classifying Involutions}, it is easy to see that for any orthogonal involution $\ddagger$ and $F$ is not characteristic $2$, $\dim(H^+) = 3$ and $\dim(H^-) = 1$. The fact that $H^+$ is three-dimensional is crucial to the proof of Theorem \ref{Class Number 1 Theorem}. On the other hand, the fact that $\dim(H^-) = 1$ allows us to make the following definition.

\begin{definition}[See \cite{Involutions}]
Let $H$ be a quaternion algebra over a field $F$ with characteristic not $2$, and with orthogonal involution $\ddagger$. The \emph{discriminant} of $\ddagger$ is
    \begin{align*}
    \disc(\ddagger) = -\nrm(\xi) \left(F^\times\right)^2 \in F^\times/\left(F^\times\right)^2
    \end{align*}
    
\noindent where $\xi \in H^- \cap H^\times$.
\end{definition}

The most important property of the discriminant is that it uniquely characterizes the involution up to isomorphism.

\begin{theorem}[See \cite{Involutions}]
Let $H$ be a quaternion algebra over a field $F$ with characteristic not $2$. Let $\ddagger_1,\ddagger_2$ be orthogonal involutions of $H$. Then the following are equivalent.
    \begin{enumerate}
        \item $(H,\ddagger_1) \cong (H,\ddagger_2)$.
        \item There exists $u \in H^\times$ such that $\tr(u) = 0$ and $z^{\ddagger_1} = uz^{\ddagger_2}u^{-1}$.
        \item $\disc(\ddagger_1) = \disc(\ddagger_2)$.
    \end{enumerate}
\end{theorem}

For local and global fields, there is another common notion of discriminant, which we must also refer to.

\begin{definition}
Let $H$ be a quaternion algebra over a local or global field $F$. Let $\nu$ be a place of $F$. We say that $H$ \emph{ramifies} at $\nu$ if $H_\nu := H \otimes_F F_\nu$ is a division algebra. If $\nu$ is an archimedian place and $H$ ramifies at $\nu$, we say that $H$ is \emph{definite} at that place. If $H$ is definite at all archimedian places, we say that it is \emph{totally definite}. Let $S$ be the set of ideals corresponding to non-archimedian places $\nu$ of $F$ such that $H$ ramifies at $\nu$. Then the \emph{discriminant} of $H$ is the ideal
    \begin{align*}
    \disc(H) = \prod_{\mathfrak{p} \in S} \mathfrak{p}.
    \end{align*}
\end{definition}

The discriminant of a quaternion algebra almost entirely classifies quaternion algebras up to isomorphism.

\begin{theorem}[See \cite{VoightBook}]\label{Discrimiant of Quaternion Algebras}
Let $H_1, H_2$ be two quaternion algebras over a local or global field $F$. The following are equivalent.
    \begin{enumerate}
        \item $H_1 \cong H_2$.
        \item $\disc(H_1) = \disc(H_2)$ and for every archimedian place $\nu$, $H_1$ is definite if and only if $H_2$ is definite.
    \end{enumerate}
\end{theorem}

Note that a corollary of Theorem \ref{Discrimiant of Quaternion Algebras} is that all totally definite quaternion algebras over a global field $F$ with an orthogonal involution are uniquely characterized as $\ddagger$-rings by $\disc(H)$ and $\disc(\ddagger)$.

\section{Orders Closed Under Involution:}

We are finally ready to introduce rings with involution that occur as orders of quaternion algebras over global fields equipped with an orthogonal involution.

\begin{definition}[See \cite{Reiner2003}]
Let $R$ be a Dedekind domain with field of fractions $F$. Let $B$ be a finite-dimensional $F$-algebra. A subring $\OO$ of $B$ is an \emph{order} if it is a finitely-generated $R$-module and $R\OO = B$---that is, $\OO$ is also a lattice. A subring $\OO$ of $B$ is a \emph{maximal order} if it is not contained in any strictly larger order of $B$. If $B$ additionally has a standard involution, we define the \emph{discriminant} of $\OO$ to be the ideal $\disc(\OO)$ such that $\disc(\OO)^2$ is the ideal generated by the set
    \begin{align*}
    \left\{\det\left(\tr(\alpha_i \alpha_j))\right)_{1 \leq i,j \leq \dim(B)}\middle| \alpha_1, \alpha_2, \ldots \alpha_{\dim(B)} \in \OO\right\}.
    \end{align*}
\end{definition}

For our purposes, we shall always work with global fields $F$ and we shall take $R = \mathfrak{o}_F$, the ring of integers of $F$. Our $F$-algebra will always be a quaternion algebra $H$. In this context, it is a classic theorem that any order $\OO$ is maximal if and only if $\disc(\OO) = \disc(H)$ \cite{Reiner2003}. However, we are interested in rings with involution specifically.

\begin{definition}[See \cite{Scharlau1974}]
Let $R$ be a Dedekind domain with field of fractions $F$. Let $B$ be a finite-dimensional $F$-algebra together with an involution $\sigma$. A subring $\OO$ of $B$ is a $\sigma$-\emph{order} if it is an order of $B$ that is closed under $\sigma$. A $\sigma$-order is a \emph{maximal} $\sigma$-\emph{order} if it is not contained inside any strictly larger $\sigma$-order.
\end{definition}

Orders closed under involutions were originally studied by Scharlau in the 1970s \cite{Scharlau1974} in the context of central simple algebras, and generalized to Azumaya algebras by Saltman \cite{Saltman1978}. To the best of the author's knowledge, there is no known classification of maximal $\sigma$-orders in such broad contexts. However, in the special case of quaternion algebras, they can be characterized by their discriminant.

\begin{theorem}[Theorem 1.1 of \cite{Sheydvasser2017}]\label{Discriminants of Maximal Orders}
Let $H$ be a quaternion algebra over a local or global field $F$ with characteristic not $2$, with orthogonal involution $\ddagger$. If $\OO$ is a $\ddagger$-order of $H$, the following are equivalent.
    \begin{enumerate}
        \item $\OO$ is a maximal $\ddagger$-order.
        \item $\disc(\OO) = \disc(H) \cap \iota(\disc(\ddagger))$, where $\iota(\disc(\ddagger))$ is the ideal generated by $\disc(\ddagger) \cap \mathfrak{o}_F$.
        \item $\OO = \OO' \cap {\OO'}^\ddagger$ for some maximal order $\OO'$ and $\disc(\OO) = \disc(H) \cap \iota(\disc(\ddagger))$.
    \end{enumerate}
\end{theorem}

Additional, stronger, results over local fields were also established by the author \cite{Sheydvasser2017}, but they will mostly be irrelevant for our purposes. We will end this section by demonstrating the relationship between an order being $\ddagger$-maximal and $\ddagger$-Euclidean.

\begin{theorem}\label{Euclidean Means Almost Maximal}
Let $H$ be a quaternion algebra over a global field $F$ with orthogonal involution $\ddagger$. Let $\OO$ be a $\ddagger$-order---it is a $\ddagger$-Euclidean order if and only if there exists a maximal $\ddagger$-order $\OO' \supset \OO$ which is a $\ddagger$-Euclidean order and ${\OO'}^+ = \OO^+$.
\end{theorem}

\begin{remark}
This result shows a clear distinction between $\ddagger$-Euclidean orders and Euclidean orders, as it is known that Euclidean orders have to be maximal---see \cite{CerriChabertPierre2013}.
\end{remark}

\begin{proof}
First, we show that if $\OO'$ is a maximal $\ddagger$-order that is also $\ddagger$-Euclidean, then any other $\ddagger$-order $\OO$ such that ${\OO'}^+ = \OO^+$ will also be $\ddagger$-Euclidean. Indeed, if $a,b \in \OO$ such that $b \neq 0$ and $ab^\ddagger \in \OO^+$, then since $\OO'$ is $\ddagger$-Euclidean, we know that there exists $q \in {\OO'}^+$ such that $\Phi(a - bq) < \Phi(b)$, where $\Phi$ is the stathm of $\OO'$. However, since ${\OO'}^+ = \OO^+$, we have thus proved that $\Phi$ is a stathm for $\OO$.

Next, suppose that $\OO$ is a $\ddagger$-Euclidean order and $\OO \subset \OO'$, where $\OO'$ is a $\ddagger$-maximal order---we wish to prove that $\OO^+ = {\OO'}^+$. First, we define
    \begin{align*}
    S = \left\{x \in \OO \backslash \{0\}\middle| x \OO' \subset \OO \right\},
    \end{align*}
    
\noindent which we note is non-empty, and therefore has an element $b$ such that $\Phi(b) \leq \Phi(x)$ for all $x \in S$. Now, for any element $q' \in {\OO'}^+$, we know that $bq' \in \OO$, and certainly $bq'b^\ddagger \in H^+$, hence we can use the fact that $\OO$ is a $\ddagger$-Euclidean order to conclude that there exists $q \in \OO^+$ such that $\Phi(bq' - bq) < \Phi(b)$. We note that clearly $(bq' - bq)\OO' \subset b\OO' \subset \OO$, hence $bq' - bq \in S \cup \{0\}$. However, by the minimality of $b$, it must be that $bq' - bq = 0$, whence $q' = q$. We thus conclude that ${\OO'}^+ = \OO^+$.

Finally, we wish to prove that if $\OO$ is $\ddagger$-Euclidean, then so is $\OO'$. Indeed, for some $k \in \mathfrak{o}_F \backslash \{0\}$, $k \OO' \subset \OO$, so we may define a function $\Phi'(x) = \Phi(kx)$. Given any $a,b \in \OO'$ such that $b \neq 0$ and $ab^\ddagger \in {\OO'}^+$, note that we can find $q \in \OO^+$ such that $\Phi(ka - kbq) < \Phi(kb)$. By definition, this means $\Phi'(a - bq) < \Phi'(b)$, hence $\Phi'$ is a stathm for $\OO'$.
\end{proof}

\section{One-Sided Ideals of Maximal Orders with Involution:}

We can now begin describing what makes maximal $\ddagger$-orders special in the context of Euclidean $\ddagger$-rings: their one-sided ideals have generators with good properties. Let $\OO$ be a maximal $\ddagger$-order $\OO$ of a quaternion algebra $H$ with orthogonal involution $\ddagger$ over a local or global field $K$. By Theorem \ref{Discriminants of Maximal Orders}, $\OO$ is hereditary, and therefore every nonzero right ideal of $\OO$ is invertible \cite{VoightBook}. Given two invertible right fractional $\OO$-ideals $I,J$, we write $I \sim J$ if and only if $\alpha I = J$ for some $\alpha \in H^\times$. We then define the right class set $\text{Cls}_R(\OO)$ as the collection of equivalence classes, i.e.
    \begin{align*}
    \text{Cls}_R(\OO) &= \left\{\text{invertible right fractional } \OO\text{-ideals}\right\}/\sim.
    \end{align*}
    
\noindent It is well-known that $\text{Cls}_R(\OO)$ is always finite---in fact, every ideal class can be represented by an integral $\OO$-ideal of bounded norm \cite{VoightBook}. Furthermore, it is known that every right ideal can always be generated by at most two elements---however, we shall need to know that we can take these generators to be of a special form. We shall first show that this is possible over almost all local fields.

\begin{lemma}\label{Local generation of ideals}
Let $H$ be a quaternion algebra over a local field $F$ with ring of integers $\mathfrak{o}_F$ such that its characteristic is not $2$ and $2 \notin \mathfrak{o}_F^\times$. Let $\ddagger$ be an orthogonal involution on $H$, let $\OO$ be a maximal $\ddagger$-order of $H$, and $I \subset \OO$ an invertible right $\OO$-ideal. If $\disc(H) = \mathfrak{p}$ or $\iota(\disc(\ddagger)) = (1)$, then there exists $x \in \OO^+$ such that $I = x \OO$.
\end{lemma}

\begin{proof}
Since $F$ is a local field, $\mathfrak{o}_F$ is a DVR, and therefore $I$ is principal \cite{VoightBook}---we only need to prove that the element it is generated by is fixed by $\ddagger$. First, suppose that $H$ is a division algebra---i.e. $\disc(H) = \mathfrak{p}$, where $\mathfrak{p}$ is the unique maximal ideal of $\mathfrak{o}_F$. Then there is a unique maximal right $\OO$-ideal
    \begin{align*}
    \mathfrak{P} = \left\{x \in \OO \middle| \nrm(x) \in \mathfrak{o}_F\right\},
    \end{align*}
    
\noindent which is, in fact, a two-sided ideal. Every other right ideal is of the form $\mathfrak{P}^n$, so if we can show that $\mathfrak{P}$ is generated by some element $x \in \OO^+$, we will have taken care of the division algebra case. Note that $\mathfrak{o}_F/\mathfrak{p} = \Bbbk$ is a finite field, and consider the quadratic form
    \begin{align*}
    q: \left(\OO^+\right) \otimes_{\mathfrak{o}_F} \Bbbk &\rightarrow \Bbbk \\
    z \otimes k &\mapsto k^2\nrm(z).
    \end{align*}
    
\noindent This is a ternary quadratic form, and since $\Bbbk$ is a finite field, it must be isotropic by the Chevalley-Warning theorem---that is, there exists $x \in \OO^+$ such that $x \notin \mathfrak{p}\OO$ and $\nrm(x) \in \mathfrak{p}$. However, $\mathfrak{P}^2 = \mathfrak{p}\OO$, so if $x \notin \mathfrak{p}\OO$, then $\nrm(x) \notin \mathfrak{p}^2$. Therefore, $x \in \mathfrak{P}/\mathfrak{P}^2$, and so we conclude that $\mathfrak{P} = x \OO$.

If $H$ is not a division algebra, then $H \cong \Mat(2,F)$. Let $I = x\OO$ for some $x \in \OO$. We want to prove that there exists $u \in \OO^\times$ such that $x':= xu \in \OO^+$. Since $\iota(\disc(\ddagger)) = (1)$, as a ring, $\OO \cong \Mat(2,\mathfrak{o}_F)$ and $\OO^\times \cong GL(2,\mathfrak{o})$. We are looking for
    \begin{align*}
    u = \begin{pmatrix} a & b \\ c & d \end{pmatrix} \in GL(2,\mathfrak{o})
    \end{align*}
    
\noindent such that $(xu)^\ddagger - xu = 0$. Note that
    \begin{align*}
    \phi: \Mat(2,F) &\rightarrow H^- \\
    u &\mapsto (xu)^\ddagger - xu
    \end{align*}
    
\noindent is a linear map from a $4$-dimensional vector space to a $1$-dimensional vector space---this implies that there are some constants $A,B,C,D \in \OO$ such that $(xu)^\ddagger - xu = 0$ if and only if $Aa + Bb + Cc + Dd = 0$. Without loss of generality, we may assume that one of $A,B,C,D \in \mathfrak{o}_F^\times$; indeed, we may assume that $A = 1$. But then it is clear that
    \begin{align*}
    u = \begin{pmatrix} C - B & 1 \\ -1 & 0 \end{pmatrix} \in GL(2,\mathfrak{o}_F)
    \end{align*}
    
\noindent and satisfies the desired condition.
\end{proof}

With this out of the way, we can get the desired result for global fields.

\begin{theorem}\label{Right Ideals Are Nice}
Let $H$ be a quaternion algebra over a global field $F$ with characteristic not $2$. Let $\ddagger$ be an orthogonal involution on $H$, let $\OO$ be a $\ddagger$-order of $H$, and let $I$ be an invertible right ideal of $\OO$. Then $I = x\OO + y\OO$ for some $x,y \in \OO$ such that $xy^\ddagger \in \OO^+$.
\end{theorem}

\begin{proof}
Let $\mathfrak{o}_F$ be the ring of integers of $F$. For every prime ideal $\mathfrak{p}$, we know that $\mathfrak{o}_{F,\mathfrak{p}}$ is a local ring and that $I_\mathfrak{p}$ is a principal ideal---recall that a lattice is invertible if and only if it is locally principal \cite{VoightBook}. Next, note that there are only finitely many prime ideals $\mathfrak{p}$ such that $\OO_\mathfrak{p}$ is not a maximal order---call this set of ideals $S$. Therefore, we can find an element $z \in \OO \cap H^\times$ such that $zI_\mathfrak{p} = (t)$ for some $t \in \mathfrak{o}_F$ for every prime ideal in $S$. With this in mind, consider the ideal $J = zI$. Since $xy^\ddagger \in \OO^+$ if and only if $\overline{x}y \in \OO^+$, it is clear that $I$ is generated by two elements $x,y \in \OO$ such that $xy^\ddagger \in \OO^+$ if and only if $J$ is. On the other hand, we know that for every prime ideal $\mathfrak{p}$, $J_\mathfrak{p}$ is generated by an element $x_\mathfrak{p} \in \OO_\mathfrak{p}^+$---this follows from Lemma \ref{Local generation of ideals} if $\mathfrak{p} \notin S$, and from the fact that $J_\mathfrak{p} = (t)$ if $\mathfrak{p} \in S$. So, if $J = x\OO + y\OO$, then for every prime ideal $\mathfrak{p}$, there exists an element $\gamma_\mathfrak{p} \in GL(2,\OO_\mathfrak{p})$ such that $(x_\mathfrak{p}, 0) = \gamma (x,y)$. In fact, it is easy to see that we can take $\gamma_\mathfrak{p} \in SL(2,\OO_\mathfrak{p})$ simply by post-multiplying by an element of the form
    \begin{align*}
    \begin{pmatrix} 1 & 0 \\ 0 & u \end{pmatrix}.
    \end{align*}
    
\noindent However, $SL(2,H)$ is the spin group of an indefinite quadratic form, and as such it is a simply-connected, absolutely almost simple algebraic group, hence we can apply the strong approximation theorem to it \cite{Kneser1965, Platonov1969}. Specifically, we can conclude that there exist elements $\gamma \in SL(2,\OO)$ that are arbitrarily close to $(\gamma_\mathfrak{p})_\mathfrak{p}$ in the $\mathfrak{a}$-adic topology. Choose a principal ideal $(t) \supset J$ with $t \in \mathfrak{o}_F$---by the above, we know that there exists $\gamma \in SL(2,\OO)$ such that $\gamma(x,y) = (x',0) \mod (t)$ for some $x' \in \OO \cap H^+$---since $(t) \supset J$, we can assume that $x' \in J$. But this is just to say that for every element $w \in J$, $w \in x'\OO + t\OO$.
\end{proof}

\begin{corollary}\label{Class Number 1 Corollary}
Let $H$ be a quaternion algebra over a global field $F$ with characteristic not $2$. Let $\ddagger$ be an orthogonal involution on $H$, and let $\OO$ be a $\ddagger$-order of $H$. If $\OO$ is $\ddagger$-Euclidean, then $\OO$ has right class number $1$.
\end{corollary}

\begin{proof}
The proof is immediate---if $I$ is a right fractional ideal of $\OO$, then by Theorem \ref{Right Ideals Are Nice}, $I = x \OO + y \OO$ for some $x,y \in \OO$ such that $xy^\ddagger \in \OO^+$. We can then apply Corollary \ref{Almost Principal Ring}.
\end{proof}

We note that Theorem \ref{Class Number 1 Theorem} is just a special case of Corollary \ref{Class Number 1 Corollary}.

\section{Classification over the Rationals:}

The fact that all Euclidean $\ddagger$-rings that arise as orders of quaternion algebras over global fields have class number $1$ has a nice consequence: we can enumerate all of them if we restrict to totally definite quaternion algebras. In particular, for rational, definite quaternion algebras, we can utilize a theorem of Brzezinski, which we briefly paraphrase below.

\begin{theorem}[\cite{Brzezinski1995, Brzezinski1998}]\label{Brzezinski's Theorem}
Let $\OO$ be an order of a rational, definite quaternion algebra $H$---then $\OO$ has right class number $1$ if and only if the integral quadratic form
    \begin{align*}
    \OO \cap H^0 &\rightarrow \ZZ \\
    z & \mapsto \nrm(z)
    \end{align*}
    
\noindent is equivalent to one of the integral quadratic forms enumerated in Table \ref{Brzezinski Table}.
\end{theorem}

\begin{table}
\begin{align*}
\begin{array}{ccccc}
\begin{pmatrix}
 1 & 1 & 1 \\
 1 & 1 & 0 \\
 1 & 0 & 1 \\
\end{pmatrix}
&
\begin{pmatrix}
 1 & 1 & 0 \\
 1 & 1 & 0 \\
 0 & 0 & 1 \\
\end{pmatrix}
&
\begin{pmatrix}
 1 & 0 & 0 \\
 0 & 1 & 0 \\
 0 & 0 & 1 \\
\end{pmatrix}
&
\begin{pmatrix}
 1 & 1 & 1 \\
 1 & 1 & 0 \\
 1 & 0 & 2 \\
\end{pmatrix}
&
\begin{pmatrix}
 1 & 0 & 1 \\
 0 & 1 & 1 \\
 1 & 1 & 2 \\
\end{pmatrix}
\\ &&&& \\
\begin{pmatrix}
 1 & 1 & 0 \\
 1 & 1 & 0 \\
 0 & 0 & 2 \\
\end{pmatrix}
&
\begin{pmatrix}
 1 & 0 & 1 \\
 0 & 1 & 0 \\
 1 & 0 & 2 \\
\end{pmatrix}
&
\begin{pmatrix}
 1 & 1 & 1 \\
 1 & 1 & 0 \\
 1 & 0 & 3 \\
\end{pmatrix}
&
\begin{pmatrix}
 1 & 0 & 0 \\
 0 & 1 & 0 \\
 0 & 0 & 2 \\
\end{pmatrix}
&
\begin{pmatrix}
 1 & 0 & 1 \\
 0 & 1 & 1 \\
 1 & 1 & 3 \\
\end{pmatrix}
\\ &&&& \\
\begin{pmatrix}
 1 & 1 & 0 \\
 1 & 2 & 2 \\
 0 & 2 & 2 \\
\end{pmatrix}
&
\begin{pmatrix}
 1 & 1 & 0 \\
 1 & 1 & 0 \\
 0 & 0 & 4 \\
\end{pmatrix}
&
\begin{pmatrix}
 1 & 1 & 1 \\
 1 & 2 & 0 \\
 1 & 0 & 2 \\
\end{pmatrix}
&
\begin{pmatrix}
 1 & 0 & 0 \\
 0 & 1 & 0 \\
 0 & 0 & 3 \\
\end{pmatrix}
&
\begin{pmatrix}
 1 & 0 & 0 \\
 0 & 2 & 2 \\
 0 & 2 & 2 \\
\end{pmatrix}
\\ &&&& \\
\begin{pmatrix}
 1 & 1 & 0 \\
 1 & 2 & 1 \\
 0 & 1 & 2 \\
\end{pmatrix}
&
\begin{pmatrix}
 1 & 0 & 0 \\
 0 & 2 & 0 \\
 0 & 0 & 2 \\
\end{pmatrix}
&
\begin{pmatrix}
 1 & 0 & 1 \\
 0 & 1 & 1 \\
 1 & 1 & 5 \\
\end{pmatrix}
&
\begin{pmatrix}
 1 & 0 & 1 \\
 0 & 2 & 2 \\
 1 & 2 & 3 \\
\end{pmatrix}
&
\begin{pmatrix}
 1 & 0 & 0 \\
 0 & 2 & 2 \\
 0 & 2 & 3 \\
\end{pmatrix}
\\ &&&& \\
\begin{pmatrix}
 1 & 0 & 1 \\
 0 & 2 & 0 \\
 1 & 0 & 3 \\
\end{pmatrix}
&
\begin{pmatrix}
 1 & 1 & 1 \\
 1 & 3 & -1 \\
 1 & -1 & 3 \\
\end{pmatrix}
&
\begin{pmatrix}
 2 & 2 & 2 \\
 2 & 2 & 0 \\
 2 & 0 & 2 \\
\end{pmatrix}
&
\begin{pmatrix}
 2 & 2 & 0 \\
 2 & 2 & 0 \\
 0 & 0 & 2 \\
\end{pmatrix}
\end{array}
\end{align*}
    \caption{All orders of rational, definite quaternion algebras with class number $1$, up to isomorphism. Orders are represented by the Gram matrix of the ternary quadratic form corresponding to them.}
    \label{Brzezinski Table}
\end{table}

\begin{remark}
It is well-known that two orders $\OO, \OO'$ are isomorphic if and only if the associated ternary quadratic forms are equivalent, so in fact, Theorem \ref{Brzezinski's Theorem} characterizes orders with class number $1$ up to isomorphism.
\end{remark}

\begin{remark}
Strictly speaking, we will only be using isomorphism classes on Brzezinski's list where the discriminant is square-free---this classification was by Vign\`{e}ras \cite{Vigneras1980}.
\end{remark}

Using Brzezinski's result, we can easily produce a table of all isomorphism classes of $\ddagger$-subrings of definite, rational quaternion algebras with orthogonal involutions. To do it, we first define an algorithm that gets us most of the way there. This algorithm will make use of the fact that the functions
    \begin{align*}
    \Psi_\OO: \left\{\text{finite subsets of } \NN\right\} &\rightarrow \mathcal{P}(\OO \cap H^0) \\
    S &\mapsto \left\{z \in \OO \cap H^0\middle| \nrm(z) \in S\right\}
    \end{align*}
    
\noindent and
    \begin{align*}
    \Omega_\OO: \OO \cap H^0 \times  &\rightarrow \mathcal{P}(\OO \cap H^0) \\
    (i,n) &\mapsto \left\{j \in \OO \cap H^0\middle|j \neq 0, \ ij = -ji, \ \nrm(j) \leq n\right\}
    \end{align*}
    
\noindent are both computable as long as $\OO$ is an order of $H$, a totally definite quaternion algebra over an algebraic number field $K$.

\begin{alg}\label{Candidate Algorithm}
On an input of an order $\OO$ of a totally definite quaternion algebra $H$ over an algebraic number field $K$ with class number $1$, this algorithm returns a list $L$ of maximal $\ddagger$-orders such that for every maximal $\ddagger$-order $\OO'$ that isomorphic (as a ring) to $\OO$, $\OO'$ is isomorphic (as a $\ddagger$-ring) to an order in $L$. Maximal $\ddagger$-orders are represented by a standard basis $1,i,j,ij$ such that $\nrm(i)$ is minimal and $z^\ddagger = (ij)\overline{z}(ij)^{-1}$.
\begin{algorithm}[H]
\begin{algorithmic}[1]
\Procedure{CorrespondingRingWithInvolutionAlg}{$\OO$}
\If{$\disc(\OO)$ is not squarefree}
    \State \textbf{return} $\{\}$
\EndIf
\State $C \gets$ generators of ideals $\mathfrak{a}$ such that $\mathfrak{a}|\disc(\OO)$
\State $I \gets \Psi_\OO\left(\left\{n \in C\middle|\disc(\OO)/\disc(H)\middle|(n)\right\}\right)$
\For{$(\xi_1, \xi_2) \in I^2$}
    \For{$x \in C$}
        \If{$x\xi_1 x^{-1} \xi_2^{-1} \in K$}
            \State $I \gets I \backslash \{\xi_2\}$
        \EndIf
    \EndFor
\EndFor
\For{$\xi \in I$}
    \State $t \gets \nrm(\xi)$
    \State $B \gets \Omega_\OO(\xi, t)$
    \While{$B = \{\}$}
        \State $t \gets 2t$
        \State $B \gets \Omega_\OO(\xi, t)$
    \EndWhile
    \State $i \gets$ element in $B$ with minimal norm
    \State $j \gets i\xi/GCD(i^2,\xi^2)$
    \State $H' = \left(\frac{i^2,j^2}{K}\right)$
    \State $\Lambda(\xi) \gets$ basis of $\OO$ in terms of basis of $H'$
\EndFor
\For{$(\xi_1, \xi_2) \in I^2$}
    \If{$\Lambda(\xi_1) = \Lambda(\xi_2)$}
        \State $I \gets I \backslash \{\xi_2\}$
    \EndIf
\EndFor
\State $L \gets \Lambda(I)$
\State \textbf{return} $L$
\EndProcedure
\end{algorithmic}
\end{algorithm}
\end{alg}

\begin{remark}
This algorithm does not guarantee that this list is minimal---i.e. some of the $\ddagger$-rings on the list can be isomorphic to one another.
\end{remark}

\begin{proof}[Proof of Correctness]
Involutions on $H$ correspond to non-zero elements $\xi \in \OO \cap H^0$ by Theorem \ref{Classifying Involutions}. We know that for any maximal $\ddagger$-order, $\disc(\OO) = \disc(H) \cap \iota(\disc(\ddagger))$ by Theorem \ref{Discriminants of Maximal Orders}---if $\disc(\OO)$ isn't squarefree, we know that it can't be a maximal $\ddagger$-order. On the other hand, $\iota(\disc(\ddagger)) = (\nrm(\xi))$ if $\xi \in \OO^-$ with minimal non-zero norm. Therefore, we know if $\mathfrak{a} = (\nrm(\xi))$, then $\mathfrak{a}|\disc(\OO)$ and $\disc(\OO)/\disc(H)|\mathfrak{a}$. Ergo, involutions on $\OO$ correspond to elements $\xi \in \OO \cap H^0$ such that if $\mathfrak{a} = (\nrm(\xi))$, then $\mathfrak{a}|\disc(\OO)$ and $\disc(\OO)/\disc(H)|\mathfrak{a}$---this is precisely what the set $I$ consists of. For each $\xi \in I$, denote the corresponding orthogonal involution by $\ddagger_{\xi}$. If $\mathfrak{p}|\disc(\OO)$, there is a unique maximal $\ddagger$-order $\OO_\mathfrak{p}$ in $H_\mathfrak{p}$ \cite{Sheydvasser2017}; therefore, if $x \in \OO \cap H^0$ has $(\nrm(x))|\disc(\OO)$, $x\OO x^{-1} = \OO$. If $\xi_1, \xi_2 \in I$ and $x\xi_1 x^{-1}\xi_2^{-1} \in K$, then the map
    \begin{align*}
    \varphi:(\OO,\ddagger_{\xi_1}) &\mapsto (\OO,\ddagger_{\xi_2}) \\
    z &\mapsto xzx^{-1}
    \end{align*}
    
\noindent is an isomorphism of $\ddagger$-rings since $\varphi(\xi_1) = x\xi_1 x^{-1} = \lambda \xi_2$ for some $\lambda \in K^\times$, hence $(\OO,\ddagger_{\xi_1})^- \mapsto (\OO,\ddagger_{\xi_2})^-$---since this subspace uniquely determines the orthogonal involution and it is evident that $\varphi$ maps $\ddagger_{\xi_1}$ to some orthogonal involution, $\varphi$ is indeed an isomorphism of $\ddagger$-rings. Therefore, if we remove from $I$ one of each pair $(\xi_1, \xi_2)$ such that $x\xi_1 x^{-1}\xi_2^{-1} \in K$, we will still have a complete list of isomorphism classes. Finally, we construct a standard basis for $(\OO,\ddagger_\xi)$ by finding an element $i \in \OO \cap H^0$ such that $i\xi = -\xi i$ and $\nrm(i)$ is minimal. Then if we define $j = i\xi/GCD(i^2,\xi^2)$ and
    \begin{align*}
    H' = \left(\frac{i^2, j^2}{K}\right),
    \end{align*}
    
\noindent then clearly $\OO \hookrightarrow H'$ and in terms of the standard basis $1,i,j,ij$, we have $(w + xi + yj + zij)^{\ddagger_{\xi}} = w + xi + yj - zij$. Rewrite $\OO$ in terms of this standard basis and remove from $I$ any duplicates. This results in the desired list $L$.
\end{proof}

Running this algorithm allows us to prove the following theorem.

\begin{theorem}\label{Definite DOrders with Class Number 1}
Let $H$ be a definite, rational quaternion algebra with orthogonal involution $\ddagger$. Let $\OO$ be a maximal $\ddagger$-order---$\OO$ has right class number $1$ if and only if it is isomorphic as a $\ddagger$-ring to one of the orders listed in Table \ref{Class Number 1 Table}.
\end{theorem}

    \begin{table}
    \begin{align*}
    \begin{array}{l|ll}
    \disc(H) = 2 & \\
    & \disc(\ddagger) = -1, H = \left(\frac{-1,-1}{\QQ}\right) & \ZZ \oplus \ZZ i \oplus \ZZ j \oplus \ZZ \frac{1 + i + j + ij}{2} \\
    & \disc(\ddagger) = -2, H = \left(\frac{-1,-2}{\QQ}\right) & \ZZ \oplus \ZZ i \oplus \ZZ \frac{1 + i + j}{2} \oplus \ZZ \frac{1 + i + ij}{2} \\
    & \disc(\ddagger) = -3, H = \left(\frac{-2,-6}{\QQ}\right) & \ZZ \oplus \ZZ i \oplus \ZZ \frac{i + j}{2} \oplus \ZZ \frac{2 + ij}{4} \\
    & \disc(\ddagger) = -5, H = \left(\frac{-1,-5}{\QQ}\right) & \ZZ \oplus \ZZ i \oplus \ZZ j \oplus \ZZ \frac{1 + i + j + ij}{2} \\
    & \disc(\ddagger) = -6, H = \left(\frac{-2,-3}{\QQ}\right) & \ZZ \oplus \ZZ i \oplus \ZZ \frac{1 + j}{2} \oplus \ZZ \frac{i + ij}{2} \\
    & \disc(\ddagger) = -10, H = \left(\frac{-1,-10}{\QQ}\right) & \ZZ \oplus \ZZ i \oplus \ZZ \frac{1 + i + j}{2} \oplus \ZZ \frac{1 + i + ij}{2} \\
    & \disc(\ddagger) = -11, H = \left(\frac{-2,-22}{\QQ}\right) & \ZZ \oplus \ZZ i \oplus \ZZ \frac{i + j}{2} \oplus \ZZ \frac{2 + ij}{4} \\
    & \disc(\ddagger) = -22, H = \left(\frac{-2,-11}{\QQ}\right) & \ZZ \oplus \ZZ \oplus \ZZ i \oplus \ZZ \frac{1 + j}{2} \oplus \ZZ \frac{i + ij}{2} \\ \hdashline
    \disc(H) = 3 & \\
    & \disc(\ddagger) = -1, H = \left(\frac{-3,-3}{\QQ}\right) & \ZZ \oplus \ZZ \frac{1 + i}{2} \oplus \ZZ j \oplus \ZZ \frac{3j + ij}{6} \\
    & \disc(\ddagger) = -3, H = \left(\frac{-1,-3}{\QQ}\right) & \begin{cases} \ZZ \oplus \ZZ i \oplus \ZZ \frac{i + j}{2} \oplus \ZZ \frac{1 + ij}{2} \\ \ZZ \oplus \ZZ i \oplus \ZZ \frac{1 + j}{2} \oplus \ZZ \frac{i + k}{2} \end{cases} \\
    & \disc(\ddagger) = -6, H = \left(\frac{-1,-6}{\QQ}\right) & \ZZ \oplus \ZZ i \oplus \ZZ \frac{1 + i + j}{2} \oplus \ZZ \frac{1 + i + ij}{2} \\ \hdashline
    \disc(H) = 5 & \\
    & \disc(\ddagger) = -2, H = \left(\frac{-5,-10}{\QQ}\right) & \ZZ \oplus \ZZ i \oplus \ZZ \frac{1 + i + j}{2} \oplus \ZZ \frac{5 + 5i + ij}{10} \\
    & \disc(\ddagger) = -5, H = \left(\frac{-2,-10}{\QQ}\right) & \ZZ \oplus \ZZ i \oplus \ZZ \frac{2 + i + j}{4} \oplus \ZZ \frac{2 + 2i + ij}{4} \\
    & \disc(\ddagger) = -10, H = \left(\frac{-2,-5}{\QQ}\right) & \ZZ \oplus \ZZ i \oplus \ZZ \frac{1 + i + j}{2} \oplus \ZZ \frac{i + ij}{2} \\ \hdashline
    \disc(H) = 7 & \\
    & \disc(\ddagger) = -1, H = \left(\frac{-7,-7}{\QQ}\right) & \ZZ \oplus \ZZ \frac{1 + i}{2} \oplus \ZZ j \oplus \ZZ \frac{7j + ij}{14} \\
    & \disc(\ddagger) = -7, H = \left(\frac{-1,-7}{\QQ}\right) & \begin{cases} \ZZ \oplus \ZZ i \oplus \ZZ \frac{i + j}{2} \oplus \ZZ \frac{1 + ij}{2} \\ \ZZ \oplus \ZZ i \oplus \ZZ \frac{1 + j}{2} \oplus \ZZ \frac{i + k}{2} \end{cases} \\ \hdashline
    \disc(H) = 13 & \\
    & \disc(\ddagger) = -13, H = \left(\frac{-2,-26}{\QQ}\right) & \ZZ \oplus \ZZ i \oplus \ZZ \frac{2 + i + j}{4} \oplus \ZZ \frac{2 + 2i + ij}{4}
    \end{array}
    \end{align*}
    \caption{All maximal $\ddagger$-orders with class number $1$, up to isomorphism (as $\ddagger$-rings). Orders are grouped by the isomorphism class of the quaternion algebra they sit inside (determined by $\disc(H)$ and $\disc(\ddagger)$). A standard basis $1,i,j,ij$ is chosen for $H$ such that $H^- = ij\QQ$ and $\nrm(i)$ is minimal.}
    \label{Class Number 1 Table}
    \end{table}
    
\begin{proof}
First, note that every maximal $\ddagger$-order of a rational, definite quaternion algebra with right class number $1$ must be isomorphic as ring to one of the orders enumerated in Theorem \ref{Brzezinski's Theorem}. Running Algorithm \ref{Candidate Algorithm} on this list produces Table \ref{Class Number 1 Table}, except that for $(\disc(H),\disc(\ddagger)) = (5,-5),(13,-13)$, the algorithm produces two orders---namely,
    \begin{align*}
    \OO_1 &= \ZZ \oplus \ZZ i \oplus \ZZ \frac{2 + i + j}{4} \oplus \ZZ \frac{2 + 2i + ij}{4} \\
    \OO_2 &= \ZZ \oplus \ZZ i \oplus \ZZ \frac{2 - i + j}{4} \oplus \ZZ \frac{2 + 2i + ij}{4}.
    \end{align*}
    
\noindent However, clearly $\OO_2 = i \OO_1 i^{-1}$, hence $(\OO_1, \ddagger) \cong (\OO_2, \ddagger)$. To see that all of the remaining orders correspond to distinct isomorphism classes of rings of involutions, note that if $(\OO_1, \ddagger_1) \cong (\OO_2, \ddagger_2)$, then $\disc(\OO_1 \otimes_\ZZ \QQ) = \disc(\OO_2 \otimes_\ZZ \QQ)$ and $\disc(\ddagger_1) = \disc(\ddagger_2)$; therefore, orders in distinct cells cannot be isomorphic to one another. For the remaining cells that contain two orders $\OO_1, \OO_2$, we note that there exist two distinct isomorphism classes of maximal $\ddagger$-orders in $H_2$ \cite{Sheydvasser2017}, hence these orders must indeed be non-isomorphic.
\end{proof}

It now remains to go through this finite list and find the $\ddagger$-rings that are $\ddagger$-Euclidean. In fact, we will find something surprising: every $\ddagger$-Euclidean order allows a very special type of Euclidean stathm.

\begin{definition}
Let $H$ be a quaternion algebra over an algebraic number field $K$ with orthogonal involution $\ddagger$. Let $\mathcal{N}_{K/\QQ}$ be the norm form from $K$ to $\QQ$. Let $\OO$ be a $\ddagger$-Euclidean order of $H$---we say that $\OO$ is \emph{norm} $\ddagger$-\emph{Euclidean} if we can take the stathm to be $\Phi(z) = \left|\mathcal{N}_{K/\QQ}(\nrm(z))\right|$.
\end{definition}

\begin{theorem}\label{Definite Euclidean DOrders}
Let $H$ be a definite, rational quaternion field with orthogonal involution $\ddagger$. Let $\OO$ be a $\ddagger$-subring of $H$---$\OO$ is $\ddagger$-Euclidean if and only if there exists a maximal $\ddagger$-order $\OO'$ that is isomorphic as a $\ddagger$-ring to one of the orders listed in Table \ref{Euclidean Ring Table}. Equivalently, $\OO$ is $\ddagger$-Euclidean if and only if it is isomorphic as a $\ddagger$-ring to either one of the orders listed in Table \ref{Euclidean Ring Table} or one of the orders listed in Table \ref{Euclidean Ring Table II}. Finally, $\OO$ is $\ddagger$-Euclidean if and only if it is norm $\ddagger$-Euclidean.
\end{theorem}

\begin{remark}
Note that Theorem \ref{Enumeration Theorem} is an immediate corollary of Theorem \ref{Definite Euclidean DOrders}.
\end{remark}

    \begin{table}
    \begin{align*}
    \begin{array}{l|ll}
    \disc(H) = 2 & \\
    & \disc(\ddagger) = -1, H = \left(\frac{-1,-1}{\QQ}\right) & \ZZ \oplus \ZZ i \oplus \ZZ j \oplus \ZZ \frac{1 + i + j + ij}{2} \\
    & \disc(\ddagger) = -2, H = \left(\frac{-1,-2}{\QQ}\right) & \ZZ \oplus \ZZ i \oplus \ZZ \frac{1 + i + j}{2} \oplus \ZZ \frac{1 + i + ij}{2} \\
    & \disc(\ddagger) = -3, H = \left(\frac{-2,-6}{\QQ}\right) & \ZZ \oplus \ZZ i \oplus \ZZ \frac{i + j}{2} \oplus \ZZ \frac{2 + ij}{4} \\
    & \disc(\ddagger) = -6, H = \left(\frac{-2,-3}{\QQ}\right) & \ZZ \oplus \ZZ i \oplus \ZZ \frac{1 + j}{2} \oplus \ZZ \frac{i + ij}{2} \\
    & \disc(\ddagger) = -10, H = \left(\frac{-1,-10}{\QQ}\right) & \ZZ \oplus \ZZ i \oplus \ZZ \frac{1 + i + j}{2} \oplus \ZZ \frac{1 + i + ij}{2} \\ \hdashline
    \disc(H) = 3 & \\
    & \disc(\ddagger) = -3, H = \left(\frac{-1,-3}{\QQ}\right) & \begin{cases} \ZZ \oplus \ZZ i \oplus \ZZ \frac{i + j}{2} \oplus \ZZ \frac{1 + ij}{2} \\ \ZZ \oplus \ZZ i \oplus \ZZ \frac{1 + j}{2} \oplus \ZZ \frac{i + k}{2} \end{cases} \\
    & \disc(\ddagger) = -6, H = \left(\frac{-1,-6}{\QQ}\right) & \ZZ \oplus \ZZ i \oplus \ZZ \frac{1 + i + j}{2} \oplus \ZZ \frac{1 + i + ij}{2} \\ \hdashline
    \disc(H) = 5 & \\
    & \disc(\ddagger) = -5, H = \left(\frac{-2,-10}{\QQ}\right) & \ZZ \oplus \ZZ i \oplus \ZZ \frac{2 + i + j}{4} \oplus \ZZ \frac{2 + 2i + ij}{4} \\
    & \disc(\ddagger) = -10, H = \left(\frac{-2,-5}{\QQ}\right) & \ZZ \oplus \ZZ i \oplus \ZZ \frac{1 + i + j}{2} \oplus \ZZ \frac{i + ij}{2} \\ \hdashline
    \disc(H) = 7 & \\
    & \disc(\ddagger) = -7, H = \left(\frac{-1,-7}{\QQ}\right) & \begin{cases} \ZZ \oplus \ZZ i \oplus \ZZ \frac{i + j}{2} \oplus \ZZ \frac{1 + ij}{2} \\ \ZZ \oplus \ZZ i \oplus \ZZ \frac{1 + j}{2} \oplus \ZZ \frac{i + k}{2} \end{cases} \\ \hdashline
    \disc(H) = 13 & \\
    & \disc(\ddagger) = -13, H = \left(\frac{-2,-26}{\QQ}\right) & \ZZ \oplus \ZZ i \oplus \ZZ \frac{2 + i + j}{4} \oplus \ZZ \frac{2 + 2i + ij}{4}
    \end{array}
    \end{align*}
    \caption{All $\ddagger$-Euclidean, maximal $\ddagger$-orders, up to isomorphism (as $\ddagger$-rings). Orders are grouped by the isomorphism class of the quaternion algebra they sit inside (determined by $\disc(H)$ and $\disc(\ddagger)$). A standard basis $1,i,j,k$ is chosen for $H$ such that $H^- = k\QQ$ and $\nrm(i)$ is minimal.}
    \label{Euclidean Ring Table}
    \end{table}
    
        \begin{table}
    \begin{align*}
    \begin{array}{l|l}
    \disc(H) = 2 & \begin{array}{ll} \disc(\ddagger) = -1, H = \left(\frac{-1,-1}{\QQ}\right) & \ZZ \oplus \ZZ i \oplus \ZZ j \oplus ij \\ \disc(\ddagger) = -3, H = \left(\frac{-2,-6}{\QQ}\right) & \ZZ \oplus \ZZ i \oplus \ZZ \frac{i + j}{2} \oplus \ZZ \frac{ij}{2} \end{array}
    \end{array}
    \end{align*}
    \caption{All $\ddagger$-Euclidean, non-maximal $\ddagger$-orders, up to isomorphism (as $\ddagger$-rings). Orders are grouped by the isomorphism class of the quaternion algebra they sit inside (determined by $\disc(H)$ and $\disc(\ddagger)$). A standard basis $1,i,j,k$ is chosen for $H$ such that $H^- = k\QQ$ and $\nrm(i)$ is minimal.}
    \label{Euclidean Ring Table II}
    \end{table}
    
\begin{proof}
By Corollary \ref{Generation of SL}, if $\OO$ is $\ddagger$-Euclidean then $SL^\ddagger(2,\OO) = E^\ddagger(2,\OO)$. In \cite{Sheydvasser2019}, it was shown that if $\OO$ is a maximal $\ddagger$-order and $\OO \cap \QQ(i)$ is a Euclidean ring, then either $\OO$ is norm $\ddagger$-Euclidean or $SL^\ddagger(2,\OO) \neq E^\ddagger(2,\OO)$; furthermore, the full list of such norm $\ddagger$-Euclidean rings was computed. This proves the claim for every case except
    \begin{align*}
    \OO = \ZZ \oplus \ZZ i \oplus \ZZ \frac{1 + i + j}{2} \oplus \ZZ \frac{5 + 5i + ij}{10} \subset \left(\frac{-5,-10}{\QQ}\right).
    \end{align*}
    
\noindent However, $\OO$ contains the subring
    \begin{align*}
    \OO' = \ZZ \oplus \ZZ i \oplus \ZZ \frac{1 + i + j}{2} \oplus \ZZ \frac{1 + i + ij}{2}
    \end{align*}
    
\noindent and it is easy to check that $\OO^\times = \{\pm 1\}$, whence we see that $E^\ddagger(2,\OO) \subset SL^\ddagger(2,\OO') \subsetneq SL^\ddagger(2,\OO)$---therefore, $\OO$ is not $\ddagger$-Euclidean. To complete the proof, one appeals to Lemma \ref{Euclidean Means Almost Maximal} and checks that $\OO$ is the smallest ring containing $\OO \cap H^+$ in every case except when $\disc(H) = 2, \disc(\ddagger) = -1,-3$. In those two cases, the smallest ring that contains $\OO^+$ is the one given in Table \ref{Euclidean Ring Table II}---as this ring is index two inside of $\OO$, there cannot be any other rings strictly contained between them.
\end{proof}

As an aside, one possible means to measure how far a $\ddagger$-order $\OO$ is from being norm $\ddagger$-Euclidean is to measure the farthest distance in $H$ to points in $\OO^+$. That is, we make the following definition.

\begin{definition}
Let $H$ be a quaternion algebra over an algebraic number field $K$ with orthogonal involution $\ddagger$. Let $\mathcal{N}_{K/\QQ}$ be the norm form from $K$ to $\QQ$, and let $\OO$ be a $\ddagger$-order of $H$. For every $p \in H^+$, define $\rho(p,\OO) = \inf_{\tau \in \OO^+} \left|\mathcal{N}_{K/\QQ}(\nrm(p - \tau))\right|$ and $\rho(\OO) = \sup_{p \in H^+} \rho(p,\OO)$.
\end{definition}

The significance of this definition is simple.

\begin{theorem}\label{Covering Measure}
Let $H$ be a quaternion algebra over an algebraic number field $K$ with orthogonal involution $\ddagger$. Let $\OO$ be a $\ddagger$-order of $H$. Then $\OO$ is a norm $\ddagger$-Euclidean order if and only if $\rho(\OO) < 1$.
\end{theorem}

\begin{proof}
The claim is an immediate consequence of the equality
\begin{align*}
\rho(\OO) &= \sup \left\{\inf_{q \in \OO^+} \left|\mathcal{N}_{K/\QQ}(\nrm(p - q))\right|\middle|p \in H^+\right\} \\
&= \sup \left\{\inf_{q \in \OO^+} \left|\mathcal{N}_{K/\QQ}(\nrm(b^{-1}a - q))\right| \middle|a,b \in \OO, \ b \neq 0, \ ab^\ddagger \in \OO^+\right\} \\
&= \sup \left\{\inf_{q \in \OO^+} \frac{\left|\mathcal{N}_{K/\QQ}(\nrm(a-bq))\right|}{\left|\mathcal{N}_{K/\QQ}(\nrm(b))\right|} \middle|a,b \in \OO, \ b \neq 0, \ ab^\ddagger \in \OO^+\right\},
\end{align*}

\noindent since the last quotient is less than $1$ if and only if $\left|\mathcal{N}_{K/\QQ}(\nrm(a-bq))\right| < \left|\mathcal{N}_{K/\QQ}(\nrm(b))\right|$, which is always possible to achieve if and only if $\OO$ is norm $\ddagger$-Euclidean.
\end{proof}

Since all of the orders $\OO$ listed in Table \ref{Euclidean Ring Table} are norm $\ddagger$-Euclidean, we know that $\rho(\OO) < 1$. These can be computed explicitly by finding the largest sphere that can be inscribed inside the lattice $\OO^+$: the results of this computation are given in Table \ref{Basic Properties Table}.

    \begin{table}
    \begin{align*}
    \begin{array}{l|l|l|l}
    H & \OO^+ & p & \rho(\OO) \\ \hline
    \left(\frac{-1,-1}{\QQ}\right) & \ZZ \oplus \ZZ i \oplus \ZZ j & \frac{1 + i + j}{2} & \frac{3}{4}\\
    \left(\frac{-1,-2}{\QQ}\right) & \ZZ \oplus \ZZ i \oplus \ZZ \frac{1 + i + j}{2} & \frac{1 + i}{2} & \frac{1}{2}\\
    \left(\frac{-2,-6}{\QQ}\right) & \ZZ \oplus \ZZ i \oplus \ZZ \frac{i + j}{2} & \frac{3 + 3i + j}{6} & \frac{11}{12} \\
    \left(\frac{-2,-3}{\QQ}\right) & \ZZ \oplus \ZZ i \oplus \ZZ \frac{1 + j}{2} & \frac{3 + 3i + j}{6} & \frac{5}{6}\\
    \left(\frac{-1,-10}{\QQ}\right) & \ZZ \oplus \ZZ i \oplus \ZZ \frac{1 + i + j}{2} & \frac{5 + 5i + 2j}{10} & \frac{9}{10} \\
    \left(\frac{-1,-3}{\QQ}\right) & \begin{cases} \ZZ \oplus \ZZ i \oplus \ZZ \frac{i + j}{2} \\ \ZZ \oplus \ZZ i \oplus \ZZ \frac{1 + j}{2} \end{cases} & \frac{3 + 3i + j}{6} & \frac{7}{12} \\
    \left(\frac{-1,-6}{\QQ}\right) & \ZZ \oplus \ZZ i \oplus \ZZ \frac{1 + i + j}{2} & \frac{3 + 3i + j}{6} & \frac{2}{3} \\
    \left(\frac{-2,-10}{\QQ}\right) & \ZZ \oplus \ZZ i \oplus \ZZ \frac{2 + i + j}{4} & \frac{1 + i}{2} & \frac{3}{4} \\
    \left(\frac{-2,-5}{\QQ}\right) & \ZZ \oplus \ZZ i \oplus \ZZ \frac{1 + i + j}{2} & \frac{5 + 5i + j}{10} & \frac{4}{5} \\
    \left(\frac{-1,-7}{\QQ}\right) & \begin{cases} \ZZ \oplus \ZZ i \oplus \ZZ \frac{i + j}{2} \\ \ZZ \oplus \ZZ i \oplus \ZZ \frac{1 + j}{2} \end{cases} & \frac{7 + 7i + 3j}{14} & \frac{23}{28} \\
    \left(\frac{-2,-26}{\QQ}\right) & \ZZ \oplus \ZZ i \oplus \ZZ \frac{2 + i + j}{4} & \frac{13 + 13i + 2j}{26} & \frac{47}{52}
    \end{array}
    \end{align*}
    \caption{A point $p$ that maximizes $\rho(p,\OO)$ for each of the $\ddagger$-orders listed in Table \ref{Euclidean Ring Table}, together with $\rho(\OO)$. Additionally, we include $\OO^\times$.}
    \label{Basic Properties Table}
    \end{table}
    
\section{Dirichlet and Coarse Fundamental Domains:}

Before we present a proof of Theorem \ref{Main Theorem}, we'll quickly recall definitions of Dirichlet and coarse fundamental domains, as well as explicitly describing the isomorphism between the group $SL^\ddagger(2,H_\RR)/\{\pm id\}$ and $\text{Isom}^0(\HH^4)$. Throughout, we shall be making use of the upper half-space model of $\HH^4$---that is,
    \begin{align*}
    \HH^4 = \left\{(x,y,z,t) \in \RR^4\middle|t > 0\right\}.
    \end{align*}

\noindent However, it will be more convenient to identify $\HH^4$ with the subset of $H_\RR$ with positive $k$-component, in which case we can write the hyperbolic distance compactly as
    \begin{align*}
    d: \HH^4 \times \HH^4 &\rightarrow \RR_{\geq 0} \\
    \left(p_1, p_2\right) &\mapsto \cosh^{-1}\left(1 + \frac{\nrm(p_1 - p_2)}{2\pi_k(p_1)\pi_k(p_2)}\right),
    \end{align*}
    
\noindent where $\pi_k(p)$ gives the $k$-th component of $p$. In this case, the boundary of $\HH^4$ is simply $H_\RR^+ \cup \{\infty\}$ and there is an action of $SL^\ddagger(2,H_\RR)$ on $\HH^4$ by M\"{o}bius transformations---that is,
    \begin{align*}
    \begin{pmatrix} a & b \\ c & d \end{pmatrix}.z = (az + b)(cz + d)^{-1}.
    \end{align*}
    
\noindent One can check that with this action, $SL^\ddagger(2,H_\RR)$ acts on $\HH^4$ by isometries and in fact $SL^\ddagger(2,H_\RR)/\{\pm id\} \cong \text{Isom}^0(\HH^4)$ \cite{Vahlen1902, Ahlfors1986}. Given any totally definite quaternion algebra $H$ with orthogonal involution $\ddagger$, one can fix an embedding $(H,\ddagger) \hookrightarrow (H_\RR, \ddagger)$ which then yields an action of $SL^\ddagger(2,H)$ on $\HH^4$---we shall always assume that we have fixed such an embedding whenever we discuss actions of groups $SL^\ddagger(2,\OO)$ on hyperbolic space. If we take $H$ to a definite, rational quaterion algebra with an orthogonal involution, and take $\OO$ to be a $\ddagger$-suborder of $H$, then $\OO$ will be a discrete subset of $H_\RR$ and correspondingly $SL^\ddagger(2,\OO)$ will be a discrete subset of $SL^\ddagger(2,H_\RR)$---ergo, $SL^\ddagger(2,\OO)/\{\pm id\}$ is a Kleinian subgroup of $\text{Isom}^0(\HH^4)$, which mean that we can define a Dirichlet domain for it.

\begin{definition}[See \cite{Engel1986,Katok1992}]
Let $(M,d)$ be a metric space. Let $\Gamma$ be a discrete group acting on $M$ by isometries. Let $z \in M$ be a point such that the stabilizer subgroup (i.e. the subgroup of $\gamma \in \Gamma$ such that $\gamma.z = z$) $\text{Stab}_\Gamma(z) = \{id\}$. The \emph{Dirichlet domain} for $\Gamma$ \emph{centered at} $z$ is the set
    \begin{align*}
    \left\{w \in M\middle| d(w,z) \leq d(w,\gamma.z) \ \forall \gamma \in \Gamma\right\}.
    \end{align*}
\end{definition}

For various important spaces such as Euclidean space $\mathbb{E}^n$ \cite{Engel1986} and hyperbolic space \cite{Katok1992}, a Dirichlet domain of a discrete group $\Gamma$ is always a \emph{fundamental domain}---that is, if $\mathcal{D}_z$ is a Dirichlet domain, then
    \begin{align*}
    \bigcup_{\gamma \in \Gamma} \gamma.\mathcal{D}_z &= M \\
    \gamma.\mathcal{D}_z^\circ \cap \mathcal{D}_z^\circ &= \emptyset, \ (\forall \gamma \neq id),
    \end{align*}
    
\noindent where $S^\circ$ is the interior of $S$. One can say more in such cases: if $\Gamma$ is an arithmetic group, then the Dirichlet domain is a finite-volume convex set bounded by finitely many geodesic faces \cite{ArithmeticGroups}. However, determining what these faces are algorithmically can still be a challenge---there are existing algorithms for Euclidean space \cite{Engel1986}, $\HH^2$ \cite{Voight2009}, and $\HH^3$ \cite{Page2015}, but there is no corresponding algorithm for $\HH^n$ with $n \geq 4$. As a first step toward handling the general $\HH^4$ case, we will just consider groups $SL^\ddagger(2,\OO)$, where $\OO$ is $\ddagger$-Euclidean. We first make use of a coarse fundamental domain as a stepping stone.

\begin{definition}[See \cite{ArithmeticGroups}]
Let $\Gamma$ be a group that acts properly discontinuously on a topological space $M$. A subset $\mathcal{F} \subset M$ is a \emph{coarse fundamental domain} for $\Gamma$ if
    \begin{enumerate}
        \item $\bigcup_{\gamma \in \Gamma} \gamma.\mathcal{F} = M$ and
        \item $\left\{\gamma \in \Gamma \middle| \gamma.\mathcal{F} \cap \mathcal{F} = \emptyset\right\}$.
    \end{enumerate}
\end{definition}

Any fundamental domain is a coarse fundamental domain, but the converse is false. We shall produce a coarse fundamental domain that is very easy to compute and such that with just a little bit of alteration, it yields a Dirichlet domain. Before we do that, however, we shall need a number of lemmas.

\begin{lemma}\label{Fudge Factor}
Let
    \begin{align*}
    g: SL^\ddagger(2,H_\RR) \times \HH^4 &\mapsto \RR \\
    \left(\left(\begin{smallmatrix} a & b \\ c & d \end{smallmatrix}\right), z\right) &\mapsto \nrm(cz + d).
    \end{align*}
    
\noindent Then $g(\gamma_1 \gamma_2, z) = g(\gamma_1, \gamma_2.z) g(\gamma_2, z)$ for all $z \in \HH^4$ and $\gamma_1, \gamma_2 \in SL^\ddagger(2,H_\RR)$.
\end{lemma}

\begin{proof}
Let
    \begin{align*}
    \gamma_i = \begin{pmatrix} a_i & b_i \\ c_i & d_i \end{pmatrix} \in SL^\ddagger(2,H_\RR),
    \end{align*}
    
\noindent so
    \begin{align*}
    \gamma_1 \gamma_2 = \begin{pmatrix} a_1 a_2 + b_1 c_2 & a_1 b_2 + b_1 d_2 \\ c_1 a_2 + d_1 c_2 & c_1 b_2 + d_1 d_2 \end{pmatrix}
    \end{align*}
    
\noindent and therefore
    \begin{align*}
    g(\gamma_1 \gamma_2, z) = \nrm\left((c_1 a_2 + d_1 c_2)z + c_1 b_2 + d_1 d_2\right).
    \end{align*}
    
\noindent On the other hand,
    \begin{align*}
    g(\gamma_1, \gamma_2.z) &= g\left(\left(\begin{smallmatrix} a_1 & b_1 \\ c_1 & d_1 \end{smallmatrix}\right), (a_2 z + b_2)(c_2 z + d_2)^{-1}\right) \\
    &= \nrm\left(c_1(a_2 z + b_2)(c_2 z + d_2)^{-1} + d_1\right) \\
    g(\gamma_2, z) &= \nrm(c_2 z + d_2),
    \end{align*}
    
\noindent ergo
    \begin{align*}
    g(\gamma_1, \gamma_2.z)g(\gamma_2, z) &= \nrm\left(c_1(a_2 z + b_2)(c_2 z + d_2)^{-1} + d_1\right)\nrm(c_2 z + d_2) \\
    &= \nrm\left(c_1(a_2 z + b_2) + d_1(c_2 z + d_2)\right) \\
    &= g(\gamma_1 \gamma_2, z),
    \end{align*}
		
\noindent as desired.
\end{proof}

The upshot of this lemma is that makes computing the $k$-th coordinate of $\gamma.z$ easy.

\begin{lemma}\label{kth coordinate}
For all $z \in \HH^4$ and $\gamma = \left(\begin{smallmatrix} a & b \\ c & d \end{smallmatrix}\right) \in SL^\ddagger(2,H_\RR)$, $\pi_k(\gamma.z) = \pi_k(z)/\nrm(cz + d)$.
\end{lemma}

\begin{proof}
Recall that $SL^\ddagger(2,H_\RR)$ is generated by elements of the form
    \begin{align*}
    \begin{pmatrix} 1 & \tau \\ 0 & 1 \end{pmatrix}, \begin{pmatrix} 0 & 1 \\ -1 & 0 \end{pmatrix}, \begin{pmatrix} u & 0 \\ 0 & \left(u^\ddagger\right)^{-1} \end{pmatrix},
    \end{align*}
    
\noindent where $\tau \in H_\RR^+$ and $u \in H_\RR^\times$ \cite{Vahlen1902,Ahlfors1986}. Note that
    \begin{align*}
    \pi_k\left(\left(\begin{smallmatrix} 1 & \tau \\ 0 & 1 \end{smallmatrix}\right).z\right) &= \pi_k(z + \tau) = \pi_k(z) \\
    \pi_k\left(\left(\begin{smallmatrix} 0 & 1 \\ -1 & 0 \end{smallmatrix}\right).z\right) &= \pi_k(-z^{-1}) = \pi_k(z)/\nrm(z) \\
    \pi_k\left(\left(\begin{smallmatrix} u & 0 \\ 0 & \left(u^\ddagger\right)^{-1} \end{smallmatrix}\right).z\right) &= \pi_k(uzu^\ddagger) = \pi_k(z),
    \end{align*}
    
\noindent hence the desired relation holds for the generators. On the other hand, suppose that $\pi_k(\gamma.z) = \pi_k(z)/g(\gamma,z)$ holds for $\gamma = \gamma_1$ and $\gamma = \gamma_2$---we claim that it then also holds for $\gamma = \gamma_1 \gamma_2$. Indeed,
    \begin{align*}
    \pi_k(\gamma_1\gamma_2.z) &= \pi_k(\gamma_2.z)/g(\gamma_1,\gamma_2.z) \\
    &= \pi_k(z)/\left(g(\gamma_1,\gamma_2.z)g(\gamma_2,z)\right) \\
    &= \pi_k(z)/g(\gamma_1\gamma_2, z).
    \end{align*}
    
\noindent Therefore, by induction, the desired relation holds for all elements of $SL^\ddagger(2,H_\RR)$.
\end{proof}

We shall eventually need to bound $\pi_k(\gamma.z)$. To accomplish this, we first obtain a trivial bound on $\nrm(c)$ given a bound on $\nrm(cz + d)$.

\begin{lemma}\label{Trivial Bound}
Let $H$ be a definite, rational quaternion algebra with orthogonal involution $\ddagger$. For any $R \in \RR^+$, $z \in \HH^4$ and $c,d \in H$ such that $cd^\ddagger \in H^+$, if $\nrm(cz + d) < R$, then $\nrm(c) < R/\pi_k(z)^2$ and $\nrm(d) < \left(\sqrt{R} + \sqrt{\nrm(cz)}\right)^2$.
\end{lemma}

\begin{proof}
If $c = 0$, the statement is obvious. Otherwise, we can write $\nrm(z + c^{-1}d) < R/\nrm(c)$. Since $cd^\ddagger \in H^+$, $c^{-1}d \in H^+$, and therefore $\nrm(z + c^{-1}d) > \pi_k(z)^2$. Ergo, $\nrm(c) < R/\pi_k(z)^2$ as desired. The second inequality follows from the observation that $\sqrt{R} > \sqrt{\nrm(cz + d)} \geq \sqrt{\nrm(d)} - \sqrt{\nrm(cz)}$.
\end{proof}

This bound has two important consequences. First, it demonstrates that one can algorithmically list all pairs $c,d \in \OO$ such that $cd^\ddagger \in \OO^+$ and $\nrm(cz + d) < R$. Second, if $z \in \HH^4$ has sufficiently large $k$-component, every point in its orbit under $SL^\ddagger(2,\OO)$ has strictly smaller $k$-component unless it is an image under a combination of Euclidean translations and rotations.

\begin{corollary}\label{Orbit goes down}
Let $H$ be a definite, rational quaternion algebra with orthogonal involution $\ddagger$. Let $\OO$ be a $\ddagger$-order of $H$. For any $z \in \HH^4$, if $\pi_k(z) > 1$ then $\pi_k(\gamma.z) \leq \pi_k(z)$ for all $\gamma \in SL^\ddagger(2,\OO)$, with equality if and only if $\gamma$ stabilizes $\infty$.
\end{corollary}

\begin{proof}
Apply Lemma \ref{Trivial Bound} with $R = \pi_k(z)$---we get that $\nrm(c) < 1/\pi_k(z) < 1$. Since $c \in \OO$, this implies that $c = 0$, hence $\gamma$ stabilizes $\infty$.
\end{proof}

We need one final lemma that will help describe some of the geodesic hyper-planes bounding our coarse fundamental domain.

\begin{lemma}\label{Perpendicular Bisector}
Let $z \in \HH^2$. If $\nrm(z) > 1$, then the perpendicular bisector to the geodesic through $z$ and $-1/z$ is given by the circle with center $z_0$ and radius $r$, where
    \begin{align*}
    z_0 &= -\frac{2\Re(z)}{\nrm(z) - 1} \\
    r &= \frac{\sqrt{\left(\nrm(z) + 1\right)^2 - 4\Im(z)^2}}{\nrm(z) - 1},
    \end{align*}
    
\noindent where $\Re(z),\Im(z)$ are the real and imaginary parts of $z$, respectively.
\end{lemma}

\begin{figure}
\begin{tabular}{cc}
\includegraphics[width=0.3\textwidth]{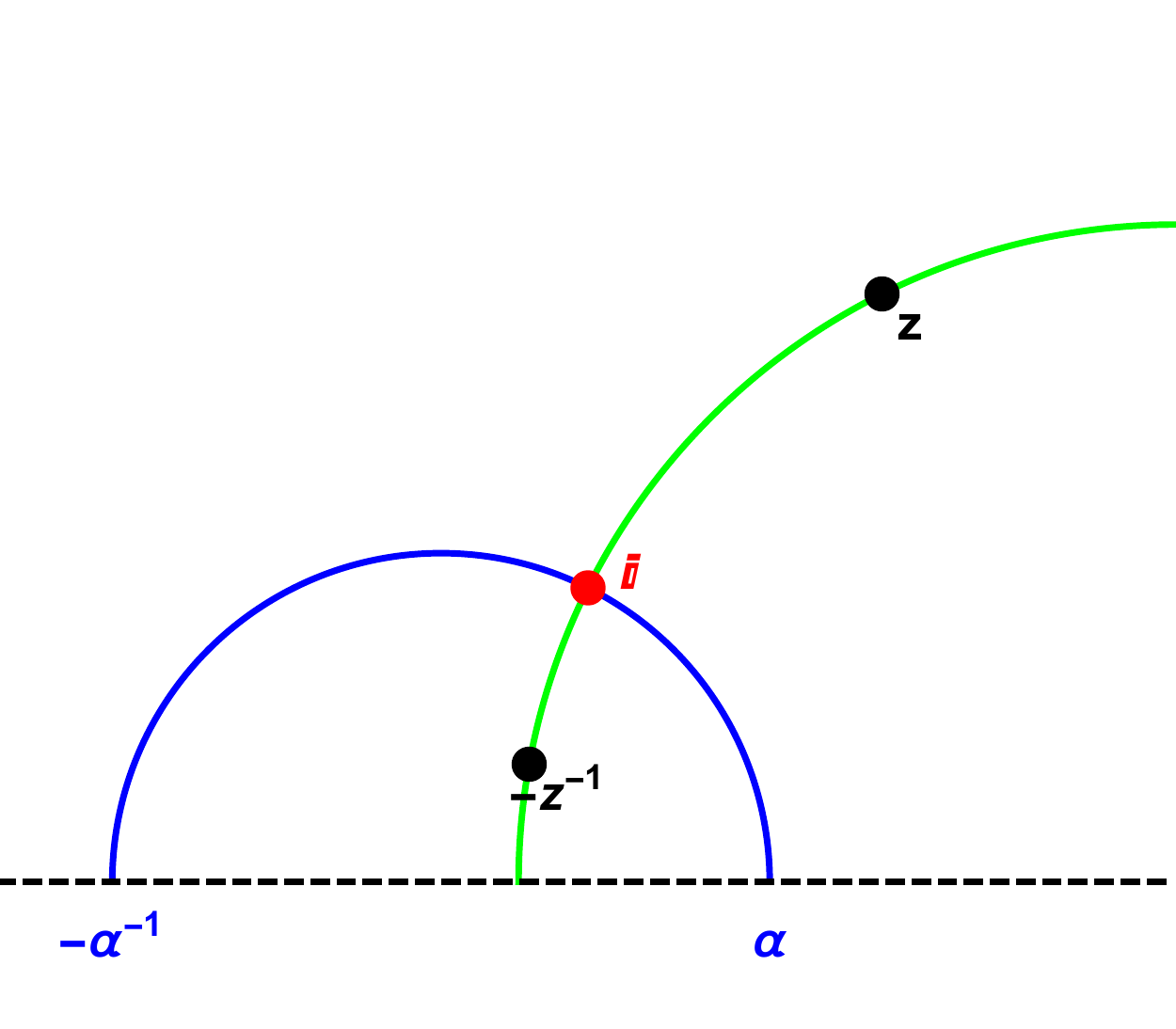} & \includegraphics[width=0.3\textwidth]{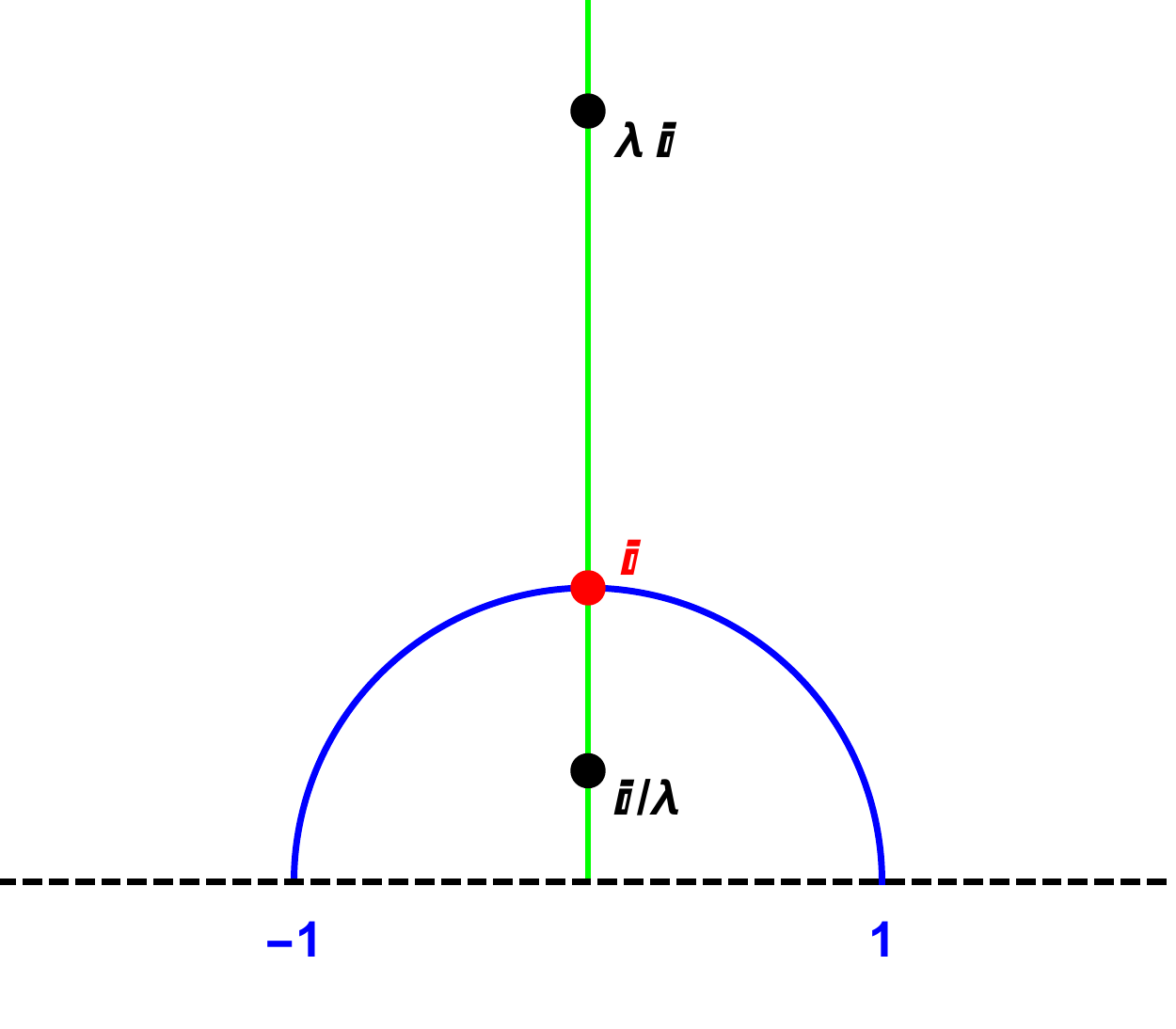}
\end{tabular}

\caption{The cross-ratio $[\alpha,-\alpha^{-1};z,-z^{-1}]$ has to equal the cross-ratio $[1,-1;\lambda i, i/\lambda]$, since these points are related by a hyperbolic rotation.}
\label{CrossRatio}
\end{figure}

\begin{proof}
Let $\alpha, \beta$ be the points where the perpendicular bisector intersects the real line---since $w \mapsto -w^{-1}$ interchanges $z$ and $-z^{-1}$, it must fix the perpendicular bisector, whence we get that $\beta = -\alpha^{-1}$. Let $l$ be the geodesic through $z$ and $-z^{-1}$; $l$ must also be fixed by $w \mapsto -w^{-1}$. Therefore, the intersection of $l$ and the perpendicular bisector is fixed, ergo it must be $i$. From this, we see that if we rotate the plane around $i$, we can move $l$ to the imaginary axis, at which point the perpendicular bisector moves to the unit circle. If this transformation moves $z \mapsto \lambda i$, then we have that the cross-ratio $[\alpha,-\alpha^{-1};z,-z^{-1}]$ has to equal the cross-ratio $[1,-1;\lambda i, i/\lambda]$, since linear fractional transformations preserve the cross-ratio. But
    \begin{align*}
    [\alpha,-\alpha^{-1};z,-z^{-1}] &= -\frac{(z-\alpha )^2}{(\alpha  z+1)^2} \\
    [1,-1;\lambda i, i/\lambda] &= -\frac{(\lambda +i)^2}{(\lambda -i)^2},
    \end{align*}
    
\noindent hence $(z-\alpha)(\lambda - i) = \pm (\lambda + i)(\alpha z + 1)$. This is easy to solve for $\alpha$ and $\lambda$ and we get
    \begin{align*}
    z_0 &= \frac{\alpha - \alpha^{-1}}{2} = -\frac{2\Re(z)}{\nrm(z) - 1} \\
    r &= \frac{\alpha + \alpha^{-1}}{2} = \frac{\sqrt{\left(\nrm(z) + 1\right)^2 - 4\Im(z)^2}}{\nrm(z) - 1},
    \end{align*}
    
\noindent as desired.
\end{proof}

\begin{theorem}\label{Coarse Fundamental Domain}
Let $H$ be a definite, rational quaternion algebra with orthogonal involution $\ddagger$. Let $H^- \otimes_\QQ \RR$ be spanned by $\xi \in \HH^4$ with $\nrm(\xi) = 1$. Let $\OO$ be a Euclidean $\ddagger$-order of $H$, $\Gamma = SL^\ddagger(2,\OO)/\{\pm id\}$, and $\Gamma_\infty = \text{Stab}_\Gamma(\infty)$. Choose an element $\alpha \in H^+$ such that $\nrm(\alpha) \leq 1$ and $\text{Stab}_{\Gamma_\infty}(\alpha) = \{id\}$. For every $t > \sqrt{2}$, the stabilizer of $z = \alpha + t \xi$ inside $\Gamma$ is trivial, and the Dirichlet domain $\mathcal{D}_z$ centered at $z$ is contained inside
    \begin{align*}
    \mathcal{D}_z^* = \left(\mathcal{F}_z \times \RR^+\right) \cap \left\{w \in \HH^4\middle| \nrm\left(w - \tau + \frac{2\alpha}{\nrm(z) - 1}\right) \geq \frac{\left(\nrm(z) + 1\right)^2 - 4t^2}{\left(\nrm(z) - 1\right)^2}, \ \forall \tau \in \OO^+\right\}
    \end{align*}
    
\noindent where $\mathcal{F}_\alpha$ is the Dirichlet domain of $\Gamma_\infty$ centered at $\alpha$. Furthermore, $\mathcal{D}_z^*$ is a coarse fundamental domain for $\Gamma$. Indeed, defining
    \begin{align*}
    L = \sup\left\{d(z,w)\middle|w \in \mathcal{D}_z^* \cap \left(\mathcal{F}_z^* \times [0,t]\right)\right\},
    \end{align*}
    
\noindent and we have that
    \begin{align*}
    \Gamma' = \left\{\gamma \in \Gamma\middle|d(\gamma.z,z) < L/2\right\}
    \end{align*}
    
\noindent is a finite set and
    \begin{align*}
    \mathcal{D}_z = \mathcal{D}_z^* \cap \left\{w \in \HH^4 \middle| d(w,z) \leq d(w,\gamma.z), \ \forall \gamma \in \Gamma'\right\}.
    \end{align*}
\end{theorem}

\begin{proof}
By Corollary \ref{Orbit goes down}, if $t > 1$, $\text{Stab}_{\Gamma_\infty}(\alpha) = \{id\}$ implies $\text{Stab}_{\Gamma}(\alpha + t\xi) = \{id\}$. By the definition of $\mathcal{D}_z$, we know that
    \begin{align*}
    \mathcal{D}_z \subset &\bigcap_{\gamma \in \Gamma_\infty} \left\{w \in \HH^4\middle|d(z,w) \leq d(\gamma.z,w)\right\} \\
    \cap&\bigcap_{\tau \in \OO^+} \left\{w \in \HH^4\middle|d(z,w - \tau) \leq d(-z^{-1},w - \tau)\right\}.
    \end{align*}
    
\noindent However, clearly
    \begin{align*}
    \bigcap_{\gamma \in \Gamma_\infty} \left\{w \in \HH^4\middle|d(z,w) \leq d(\gamma.z,w)\right\} = \mathcal{F}_z \times \RR^+ and
    \end{align*}
    
    \begin{align*}
    \left\{w \in \HH^4\middle|d(z,w - \tau) \leq d(-z^{-1},w - \tau)\right\} = \left\{w \in \HH^4\middle|d(z,w) \leq d(-z^{-1},w)\right\} + \tau.
    \end{align*}
    
\noindent The collection of all $w$ satisfying $d(z,w) = d(-z^{-1},w)$ is the geodesic hyper-sphere orthogonal to the geodesic between $z$ and $-z^{-1}$---by Lemma \ref{Perpendicular Bisector}, we know that this is the set
    \begin{align*}
    \left\{w \in \HH^4\middle| \nrm\left(w + \frac{2\alpha}{\nrm(z) - 1}\right) = \frac{\left(\nrm(z) + 1\right)^2 - 4t^2}{\left(\nrm(z) - 1\right)^2}\right\}.
    \end{align*}
    
\noindent The radius of this hyper-sphere is
    \begin{align*}
    R:= \frac{\sqrt{\left(\nrm(z) + 1\right)^2 - 4t^2}}{\nrm(z) - 1} &= \frac{1}{1 + \frac{\nrm(\alpha) - 1}{t^2}}\sqrt{1 + \frac{2(\nrm(\alpha) - 1)}{t^2} + \frac{(1 + \nrm(\alpha))^2}{t^4}},
    \end{align*}
    
\noindent and since $0 \leq \nrm(\alpha) \leq 1$, it is easy to check that
    \begin{align*}
    1 \leq R \leq \sqrt{1 + \frac{4}{t^4}}.
    \end{align*}
    
\noindent It follows that if $t > \sqrt{2}$, then $\alpha + t\xi$ will lie above all translations of this geodesic hyper-sphere. In that case, we will have
    \begin{align*}
    &\left\{w \in \HH^4\middle|d(z,w - \tau) \leq d(-z^{-1},w - \tau)\right\} \\
    &= \left\{w \in \HH^4\middle| \nrm\left(w - \tau + \frac{2\alpha}{\nrm(z) - 1}\right) \geq \frac{\left(\nrm(z) + 1\right)^2 - 4t^2}{\left(\nrm(z) - 1\right)^2}\right\}
    \end{align*}
    
\noindent for all $\tau \in \OO^+$. Thus, we see that
        \begin{align*}
        \mathcal{D}_z^* = &\bigcap_{\gamma \in \Gamma_\infty} \left\{w \in \HH^4\middle|d(z,w) \leq d(\gamma.z,w)\right\} \\
    \cap&\bigcap_{\tau \in \OO^+} \left\{w \in \HH^4\middle|d(z,w - \tau) \leq d(-z^{-1},w - \tau)\right\}.
        \end{align*}
        
\noindent Next, for every point $p \in \mathcal{F}_z$, the smallest $t$ such that $p + t \xi \in \mathcal{D}_z^*$ will be given by
    \begin{align*}
    \sup\left\{t \in \mathbb{R}^+\middle|\exists \tau \in \OO^+, \ d(z,p + t \xi - \tau) = d(-z^{-1},p + t \xi - \tau)\right\}.
    \end{align*}
    
\noindent To see that this is set is always non-empty, regardless of the choice of $p$, consider the projection of the set
    \begin{align*}
    \bigcup_{\tau \in \OO^+} \left\{w \in \HH^4\middle|d(z,w - \tau) = d(-z^{-1},w - \tau)\right\}
    \end{align*}
    
\noindent onto $H^+$---this will be a union of translations of a disk with radius at least $1$---since by Theorems \ref{Definite Euclidean DOrders} and \ref{Covering Measure} $\rho(\OO) < 1$, we know that these disks cover $H^+$. Ergo, for every $p$, there exists $s$ such that $p + s\xi$ satisfies $d(z,p + s\xi - \tau) = d(-z^{-1},p + s\xi - \tau)$ for some $\tau \in \OO^+$. Note that
    \begin{align*}
    \inf&\left\{s \in \RR^+\middle|\exists p \in \mathcal{F}_z, \ p + s \xi \in \mathcal{D}_z^*\right\} \\
    &= \inf_{p \in \mathcal{F}_z}\sup\left\{s \in \RR^+\middle|\exists \tau \in \OO^+, \ d(z,p + s \xi - \tau) = d(-z^{-1},p + s \xi - \tau)\right\},
    \end{align*}
    
\noindent which is maximized when $p$ is as far away from
    \begin{align*}
    \frac{2\alpha}{\nrm(z) - 1} + \OO^+
    \end{align*}
    
\noindent as possible. It is easy to see that in this case,
    \begin{align*}
    \inf_{\tau \in \OO^+} \nrm\left(p - \tau + \frac{2\alpha}{\nrm(z) - 1}\right) = \rho(\OO)
    \end{align*}
    
\noindent and
    \begin{align*}
    \nrm\left(p - \tau + \frac{2\alpha}{\nrm(z) - 1} + s\xi\right) &= \rho(\OO) + s^2,
    \end{align*}
    
\noindent whence $s = \sqrt{R^2 - \rho(\OO)}$. Thus
    \begin{align*}
    \inf\left\{s \in \RR^+\middle|\exists p \in \mathcal{F}_z, \ p + s \xi \in \mathcal{D}_z^*\right\} = \sqrt{R^2 - \rho(\OO)} \geq \sqrt{1 - \rho(\OO)}.
    \end{align*}
    
\noindent Next, choose any $\gamma \in \Gamma \backslash \{id\}$. Let $\overline{l}_\gamma$ be the geodesic line segment from $z$ to $\partial \mathcal{D}_z^*$ in the direction of $\gamma.z$. Let $l_\gamma$ be the length of this line segment. By the definition of $\mathcal{D}_z$, $\gamma.\mathcal{D}_z$ intersects $\mathcal{D}_z^*$ only if $d(z,\gamma z) \leq l_\gamma/2$. By Corollary \ref{Orbit goes down}, we know that all geodesic line segments $\overline{l}_\gamma$ are contained inside of $\mathcal{D}_z^* \cap \left(\mathcal{F}_z^* \times [0,t]\right)$. But this set is contained inside $\mathcal{F}_z^* \times [\sqrt{1 - \rho(\OO)},t]$ which is compact---therefore,
    \begin{align*}
    L = \sup\left\{d(z,w)\middle|w \in \mathcal{D}_z^* \cap \left(\mathcal{F}_z^* \times [0,t]\right)\right\}
    \end{align*}
    
\noindent exists and the only $\gamma \in \Gamma$ such that $\gamma.\mathcal{D}_z$ intersects $\mathcal{D}_z^*$ satisfy $d(\gamma.z,z) \leq L/2$. This collection is $\Gamma'$---since $\Gamma$ is a discrete group, $\Gamma'$ is finite, hence $\mathcal{D}_z^*$ is a coarse fundamental domain. That
    \begin{align*}
    \mathcal{D}_z = \mathcal{D}_z^* \cap \left\{w \in \HH^4 \middle| d(w,z) \leq d(w,\gamma.z), \ \forall \gamma \in \Gamma'\right\}
    \end{align*}
    
\noindent is also an immediate consequence.
\end{proof}

As a corollary, we get Theorem \ref{Main Theorem}.

\begin{proof}[Proof of Theorem \ref{Main Theorem}]
Given a Euclidean $\ddagger$-order $\OO$, compute the collection of $\gamma \in \Gamma$ such that $d(\gamma.z,z) < L/2$, as defined in Theorem \ref{Coarse Fundamental Domain}. This is possible since the bound $d(\gamma.z,z) < L/2$ implies a bound on $\nrm(cz + d)$, and we know by Lemma \ref{Trivial Bound} that there are only finitely many $c,d$ such that $\nrm(cz + d) < k$ for any $k$, and these $c,d$ can be effectively enumerated. We can then use Algorithm \ref{Euclidean Algorithm} to compute corresponding $a,b$ for each pair of $c,d$. By Theorem \ref{Coarse Fundamental Domain}, this collection $\Gamma'$ defines a Dirichlet domain---specifically,
    \begin{align*}
    \mathcal{D}_z \subset &\bigcap_{\gamma \in \Gamma_\infty} \left\{w \in \HH^4\middle|d(z,w) \leq d(\gamma.z,w)\right\} \\
    \cap&\bigcap_{\tau \in \OO^+} \left\{w \in \HH^4\middle|d(z,w - \tau) \leq d(-z^{-1},w - \tau)\right\},
    \end{align*}
    
\noindent which is a computable hyperbolic polytope with finitely many sides. That this polytope has only one cusp is evident from the fact that $\mathcal{D}_z^*$ only has one cusp.
\end{proof}

\section{Open Problems and Counter-Examples:}

We conclude with some discussion of existing open problems, as well as some counter-examples to obvious conjectures. We will begin with the big picture.

\begin{problem}\label{First Problem}
Do there exist interesting examples of $\sigma$-Euclidean rings other than orders of quaternion algebras with orthogonal involution? For example, for what other (non-trivial) classes of $R$ are all right invertible ideals of the form $I = xR + yR$ with $x\sigma(y) \in R^+$ for some involution $\sigma$?
\end{problem}

\begin{remark}
It is possible that the requirement that every right invertible ideal be of the form $I = xR + yR$ with $x\sigma(y) \in R^+$ is a stronger requirement than necessary: for $R$ being $\sigma$-Euclidean to imply that it is a principal ring, it suffices for every right invertible ideal to admit a minimal set of generators such that least two of them $x,y$ satisfy $x\sigma(y) \in R^+$.
\end{remark}

\begin{remark}
Theorem \ref{Right Ideals Are Nice} likely applies in a much larger context than just orders of quaternion algebras, for the following reason: ideals of orders of quaternion algebras are generated by two elements in general and if $R^+$ is large, it should be possible to choose those generators $a,b$ carefully so that $a\sigma(b) \in R^+$.
\end{remark}

\begin{problem}\label{Second Problem}
If $R$ is a Euclidean ring, does there exist some involution $\sigma$ such that $R$ is $\sigma$-Euclidean? Is this true if we restrict to the case where $R$ is an order of a quaternion algebra with (orthogonal) involution?
\end{problem}

\begin{remark}\label{Euclidean List}
All Euclidean rings that are orders of definite, rational quaternion algebras are isomorphic to one of the following (see \cite{Vigneras1980}).
	\begin{align*}
	\ZZ \oplus \ZZ i \oplus \ZZ j \oplus \ZZ \frac{1 + i + j + ij}{2} &\subset \left(\frac{-1,-1}{\QQ}\right) \\
  \ZZ \oplus \ZZ i \oplus \ZZ \frac{1 + i + j}{2} \oplus \ZZ \frac{1 + i + ij}{2} &\subset \left(\frac{-1,-2}{\QQ}\right) \\
	\ZZ \oplus \ZZ i \oplus \ZZ \frac{1 + i + j}{2} \oplus \ZZ \frac{i + ij}{2} &\subset \left(\frac{-2,-5}{\QQ}\right).
	\end{align*}
	
\noindent From Theorem \ref{Enumeration Theorem}, we know that each of these can be given an orthogonal involution $\ddagger$ such that they are $\ddagger$-Euclidean. Does this continue to hold true for all orders of totally definite quaternion algebras? All orders of quaternion algebras over global fields? All rings with involution?
\end{remark}

\begin{remark}
It is not true that if a Euclidean ring is $\sigma$-Euclidean with one choice of $\sigma$, then it is $\sigma$-Euclidean with any choice of $\sigma$. As an example,
	\begin{align*}
	\OO_1 = \ZZ \oplus \ZZ i \oplus \ZZ \frac{i + j}{2} \oplus \ZZ \frac{1 + ij}{2} &\subset \left(\frac{-1,-3}{\QQ}\right)
	\end{align*}
	
\noindent is isomorphic (as a ring) to
	\begin{align*}
	\OO_2 = \ZZ \oplus \ZZ \frac{1 + i}{2} \oplus \ZZ j \oplus \ZZ \frac{3j + ij}{6} &\subset \left(\frac{-3,-3}{\QQ}\right)
	\end{align*}
	
\noindent since they are both maximal orders of a definite quaternion algebra over $\QQ$ with discriminant $3$, and the class number is $1$. Furthermore, they are both Euclidean as per Remark \ref{Euclidean List}. But if we define $z^\ddagger = (ij)\overline{z}(ij)^{-1}$, then $\disc(\ddagger) = -3$ for the first one, and $-1$ for the second; therefore, by Theorem \ref{Definite Euclidean DOrders}, $\OO_1$ is $\ddagger$-Euclidean, but $\OO_2$ is not.
\end{remark}

\begin{problem}
Let $H$ be a quaternion algebra over an algebraic number field $K$ with orthogonal involution $\ddagger$. Suppose that some $\ddagger$-order of $H$ is $\ddagger$-Euclidean. Is every maximal $\ddagger$-order of $H$ $\ddagger$-Euclidean?
\end{problem}

\begin{remark}
The corresponding statement for Euclidean orders is true: if one order is Euclidean, then it is maximal and it has class number $1$---it follows that all orders are isomorphic \cite{VoightBook} and therefore all maximal orders are Euclidean. For the ring with involution case, by Theorem \ref{Euclidean Means Almost Maximal}, if $H$ contains a $\ddagger$-Euclidean order, we know that $H$ certainly contains maximal $\ddagger$-orders that are $\ddagger$-Euclidean. Furthermore, any maximal $\ddagger$-order isomorphic (as a ring with involution) to a $\ddagger$-Euclidean order is obviously $\ddagger$-Euclidean. However, even if the right class number is $1$, it does not follow that all maximal $\ddagger$-orders of $H$ are isomorphic as rings with involution. As a counter-example, consider
	\begin{align*}
	\OO_1 &= \ZZ \oplus \ZZ i \oplus \ZZ \frac{i + j}{2} \oplus \ZZ \frac{1 + ij}{2} \subset \left(\frac{-1,-3}{\QQ}\right) \\
	\OO_2 &= \ZZ \oplus \ZZ i \oplus \ZZ \frac{1 + j}{2} \oplus \ZZ \frac{i + k}{2} \subset \left(\frac{-1,-3}{\QQ}\right).
	\end{align*}
	
\noindent Both are maximal $\ddagger$-orders if $z^\ddagger = (ij)\overline{z}(ij)^{-1}$. However, they are not isomorphic as rings with involution, since $\tr(\OO_1^+) = (2)$ and $\tr(\OO_2^+) = (1)$.
\end{remark}

\begin{problem}
Enumerate all isomorphism classes of $\ddagger$-Euclidean rings that arise as $\ddagger$-orders of totally definite quaternion algebras with orthogonal involutions.
\end{problem}

\begin{remark}\label{Finite Remark}
If $H$ is a totally definite quaternion algebra, then the underlying algebraic number field $K$ must be totally real. It is known that the number of totally real number fields $K$ with discriminant less than $\leq C$ is finite \cite{Odlyzko1990}, whence the number of Eichler orders of totally definite quaternion algebras with class number $1$ is finite---these were in fact enumerated by Kirschmer and Voight \cite{KirschmerVoight2010}. On any such Eichler order, there are only finitely many different choices of orthogonal involution $\ddagger$. By Theorem \ref{Discriminants of Maximal Orders}, and Corollary \ref{Class Number 1 Corollary}, every maximal $\ddagger$-order that is $\ddagger$-Euclidean is an Eichler order with class number $1$. By Theorem \ref{Euclidean Means Almost Maximal}, there are only finitely many $\ddagger$-Euclidean orders contained in any given maximal $\ddagger$-order. We conclude that the number of isomorphism classes is finite; it remains to actually enumerate all of them.
\end{remark}

\begin{problem}
Do there exist infinitely many isomorphism classes of Euclidean $\ddagger$-orders of quaternion algebras over algebraic number fields with orthogonal involution?
\end{problem}

\begin{remark}
In light of Remark \ref{Finite Remark}, we know that if this is true, then all but finitely many of these isomorphism classes correspond to indefinite quaternion algebras. On the other hand, if $H$ is totally indefinite and the base field has Euclidean ring of integers, then maximal orders of $H$ are Euclidean \cite{CerriChabertPierre2014}---consequently, there are infinitely many isomorphism classes of Euclidean rings that arise as maximal orders of quaternion algebras over algebraic number fields. It is possible that a similar result is true for $\ddagger$-orders.
\end{remark}

\begin{problem}
Does there exist a non-commutative $\sigma$-Euclidean ring such that no choice of stathm is multiplicative?
\end{problem}

\begin{remark}
It is known due to the work of Conidis, Nielsen, and Tombs that there exist Euclidean domains where no choice of stathm is multiplicative \cite{ConidisNielsenTombs2019}---therefore, if we allow the ring to be commutative, then it is easy to answer this question in the negative simply by taking $\sigma$ to be the identity function.
\end{remark}

\begin{problem}
Let $\OO$ be a $\ddagger$-order of a definite, rational quaternion algebra with orthogonal involution. If $\OO$ is not $\ddagger$-Euclidean, does it follow that $E^\ddagger(2,\OO) \neq SL^\ddagger(2,\OO)$? Is $E^\ddagger(2,\OO)$ an infinite index subgroup? Is it a non-normal subgroup?
\end{problem}

\begin{remark}
In all cases where the answer is known, if $\OO$ is not $\ddagger$-Euclidean, then $E^\ddagger(2,\OO)$ is an infinite-index subgroup of $SL^\ddagger(2,\OO)$---specifically, this is true if $\OO$ is a maximal $\ddagger$-ring and $\OO^+$ contains a Euclidean subring \cite{Sheydvasser2019}. If this is true generally, then it is analogous to a phenomenon first noted by Cohn \cite{Cohn1966}: if $K$ is an imaginary quadratic field and $\mathfrak{o}_K$ is its ring of integers, then either $\mathfrak{o}_K$ is Euclidean or $SL(2,\mathfrak{o}) \neq E(2,\mathfrak{o})$. In fact, it is true that $\mathfrak{o}$ is any order of $K$, then either $\rho(\mathfrak{o}) \leq 1$ or $E(2,\mathfrak{o})$ is an infinite-index, non-normal subgroup of $SL(2,\mathfrak{o})$---this was proved by elementary means by Nica \cite{Nica2011}, with some minor corrections given by the author \cite{Sheydvasser2016}.
\end{remark}

\begin{problem}
Given a maximal $\ddagger$-order $\OO$ of a rational, definite quaternion algebra with orthogonal involution, can one compute the volume of the orbifold $\HH^4/SL^\ddagger(2,\OO)$ in terms of algebraic invariants of $\OO$?
\end{problem}

\begin{problem}
Produce an algorithm such that, given a maximal $\ddagger$-order $\OO$ of a rational, definite quaternion algebra with orthogonal involution, computes a Dirichlet domain for the group $SL^\ddagger(2,\OO)$.
\end{problem}

\begin{problem}
If $(R,\sigma)$ is a ring with involution and a division algebra, then $SL^\sigma(2,R)$ acts transitively on $R^+$. What interesting topological spaces arise this way?
\end{problem}

\subsection*{Acknowledgment.}
The author would like to thank Alex Kontorovich, Jeffrey Lagarias, and John Voight for helpful discussion and feedback.

\bibliography{EuclideanBib}
\bibliographystyle{alpha}
\end{document}